\documentclass[oneside,english]{amsart}
\usepackage[T1]{fontenc}
\usepackage[utf8]{inputenc}
\usepackage[a4paper]{geometry}
\usepackage{babel}
\usepackage{float}
\usepackage{amstext}
\usepackage{amsthm}
\usepackage{amssymb}
\usepackage{graphicx}
\usepackage{todonotes}
\usepackage[unicode=true,pdfusetitle,
 bookmarks=true,bookmarksnumbered=false,bookmarksopen=false,
 breaklinks=false,pdfborder={0 0 1},backref=false,colorlinks=false]
 {hyperref}

\newcommand{\modifa}[1]{#1}
\newcommand{\modifb}[1]{#1}

\makeatletter

\providecommand{\tabularnewline}{\\}
\newcommand{\lyxdot}{.}

\numberwithin{equation}{section}
\numberwithin{figure}{section}
\theoremstyle{plain}
\newtheorem{thm}{\protect\theoremname}
\theoremstyle{remark}
\newtheorem{rem}[thm]{\protect\remarkname}
\theoremstyle{definition}
\newtheorem{defn}[thm]{\protect\definitionname}

\makeatother

\providecommand{\definitionname}{Definition}
\providecommand{\remarkname}{Remark}
\providecommand{\theoremname}{Theorem}

\global\long\def\Pu{{p_1}}
\global\long\def\Pv{{p_2}}
\global\long\def\Pw{{p_3}}
\global\long\def\Px{{p_4}}
\global\long\def\Py{{p_5}}
\global\long\def\Pz{{p_6}}
\global\long\def\Pa{{p_7}}
\global\long\def\Pb{{p_8}}
\global\long\def\Pc{{p_9}}
\global\long\def\Pp{{p_{10}}}

\global\long\def\a{a}
\global\long\def\b{b}

\global\long\def\geq{\geqslant}
\global\long\def\leq{\leqslant}

\global\long\def\LB{{\text{LB}}}
\global\long\def\eqeq{{\text{EqEq}}}
\global\long\def\eqsys{{\text{EqSys}}}

\begin{document}
\title{Stability analysis of the vectorial lattice-Boltzmann method}
\author{Kévin Guillon}\address{Institut de Mathématiques de Bordeaux, CNRS, UMR 5251, Bordeaux, France}
\author{Romane Hélie, Philippe Helluy}\address{Université de Strasbourg, CNRS, IRMA UMR 7501, Inria Tonus, Strasbourg, France}


\begin{abstract}
We perform a stability analysis of the Vectorial Lattice-Boltzmann
Method (VLBM). The VLBM has been introduced in \cite{bouchut1999construction,aregba2000discrete,dubois2014entropy,coulette2019high,baty2023robust}.
It is a variant of the LBM with extended stability features:
it allows handling compressible flows with shock waves, while the
LBM is limited to low-Mach number regime. The stability analysis is
based on the Legendre transform theory. We also propose a new tool:
the equivalent system analysis, which we conjecture to contains both
the stability and the consistency of the VLBM. 
\end{abstract}


\maketitle

\section{Introduction}

The Vectorial Lattice-Boltzmann Method (VLBM) is a variation of the Lattice-Boltzmann Method (LBM) proposed by several authors \cite{bouchut1999construction,aregba2000discrete,dellar2002lattice,dubois2014entropy,coulette2019high,baty2023robust} for solving systems of conservation laws. Compared to the traditional LBM, which uses a scalar distribution function, the VLBM utilizes a vectorial distribution function and offers several advantages. For instance, it allows for a fully rigorous entropy stability analysis \cite{bouchut1999construction,dubois2014entropy}, which ensures that the resulting schemes can compute shock waves \cite{baty2023robust} while the standard LBM is only stable on low-Mach number flows.

Both the LBM and the VLBM are based on a kinetic representation of the system of conservation laws, which consists of a set of transport equations coupled through a stiff relaxation term with a small parameter $\varepsilon$. Solving the transport and relaxation separately through a splitting method results in simple and efficient schemes.

However, this methodology raises several natural questions, three of which are listed below:

\begin{enumerate}
\item Is the kinetic representation stable as $\varepsilon \to 0$?
\item Is the kinetic representation consistent with the system of conservation laws as $\varepsilon \to 0$?
\item Is the splitting scheme consistent with the kinetic representation and/or the system of conservation laws as $\varepsilon \to 0$?
\end{enumerate}

The first question can be answered by an entropy analysis, which has been explored in prior work \cite{bouchut1999construction,dubois2014entropy}. However, the fully rigorous theory has not yet been achieved, especially for general systems of conservation laws in high-dimensional spaces. This is related to challenging questions, such as the mathematical theory of Navier-Stokes equations \cite{saint2009hydrodynamic}, strange behavior of the general solutions of conservation laws \cite{de2012h}, and uncertainty quantification \cite{fjordholm2017construction}.

The second question regarding whether or not the kinetic model converges towards the approximated conservation laws can be fully answered in some specific cases. However, the general case is also difficult. There are some semi-heuristic arguments, based on the Chapman-Enskog expansion \cite{chapman1990mathematical}, that provide explanations for the convergence of the stiff relaxation approximation. Among many works in this direction, \cite{dellar2001bulk} presents one approach.

The third question is related to the fact that the stiff relaxation is never resolved in practice. Instead, a time splitting scheme with over-relaxation is used. In most situations, this splitting algorithm does not provide a good approximation of the relaxation model, but rather it approximates the initial conservation laws directly, which is the ultimate objective. Theoretically, this approach is justified with an equivalent equation analysis, performed directly on the time splitting scheme without the relaxation intermediary. The analysis relies on a combination of Taylor and Chapman-Enskog expansions \cite{dellar2013interpretation} and generally provides relevant information on the consistency of the LBM or VLBM method.

However, the stability of the equivalent equation is generally not sufficient to ensure the stability of the VLBM \cite{bouchut2005stability}. Thus, separate studies are necessary for both the stability and the consistency of the method in practice.

In this work, we aim to provide a more comprehensive justification for the stability and consistency of the VLBM method. We begin by reviewing the entropy stability of the VLBM, which can be proven using the Legendre transform theory. We demonstrate that this approach can be naturally extended to prove the stability of the VLBM when an over-relaxed splitting scheme is used for time integration.

To analyze consistency, we propose a new algorithmic approach for constructing a system of equivalent Partial Differential Equations (PDEs) that includes not only the conservative variables but also the equilibrium deviation variables. In several examples, we show that the analysis of the equivalent system and the entropy analysis yield the same stability condition.

We then utilize an additional Chapman-Enskog expansion to remove the equilibrium deviation variables, under a smallness hypothesis. This enables us to recover the traditional equivalent equation provided by other authors \cite{dubois2008equivalent,graille2014approximation,courtes2020vectorial}.

Finally, we present numerical experiments to verify that the equivalent system with deviation variables provides more accurate information about the stability of the VLBM than the standard equivalent equation.

\section{Hyperbolic conservation law and duality}

\subsection{Hyperbolic systems of conservation laws}

In this work, we consider the numerical resolution of a system of
conservation laws
\begin{equation}
\partial_{t}W+\sum_{i=1}^{d}\partial_{i}Q^{i}(W)=0,\label{eq:conslaw}
\end{equation}
where the unknown is a vector of $m$ conserved quantities $W(X,t)\in\mathbb{R}^{m}$,
depending on a space variable $X\in\mathbb{R}^{d}$ and a time variable
$t\in\mathbb{R}$. Let $N=(N_{1},\ldots,N_{d})\in\mathbb{R}^{d}$
be a space vector. The flux of the system is defined by
\[
Q(W,N)=\sum_{i=1}^{d}Q^{i}(W)N_{i}.
\]
We assume that the system of conservation laws admits a Lax entropy
$W\mapsto s(W)$. Therefore, there is an entropy flux $\sum g^{i}(W)N_{i}$
such that
\[
\partial_{t}s(W)+\sum_{i=1}^{d}\partial_{i}g^{i}(W)=0,
\]
whenever $W$ is a smooth solution of (\ref{eq:conslaw}). This imposes
that
\begin{equation}
D_{W}s(W)D_{W}Q^{i}(W)=D_{W}g^{i}(W),\label{eq:entropyfluxrelation}
\end{equation}
where we have denoted by $D_{W}g(W)$ the Jacobian of $g(W)$. Let
us recall that the Jacobian is the transpose of the gradient
\begin{equation}
D_{W}s(W)=\nabla_{W}s(W)^{\intercal}.\label{eq:transpose}
\end{equation}
Thus $D_{W}s(W)$ is a row vector, while $\nabla_{W}s(W)$ is a column
vector.

In addition, we assume a convexity hypothesis of $s$ on a closed
convex cone $\mathcal{K}\subset\mathbb{R}^{m}$:
\[
s\text{ is strictly convex on }\mathcal{K}.
\]
 Finally, following a standard convex analysis convention, we extend
$s$ by an infinite value outside $\mathcal{K}$:
\begin{equation}
s(W)=+\infty,\quad W\in\complement\mathcal{K}.\label{eq:infty_outside}
\end{equation}

\begin{rem}
For working in a fully practical framework, it is indeed necessary
to consider cases where $\mathcal{K}\neq\mathbb{R}^{d}$. However,
this leads to mathematical questions that we have not yet investigated.
Therefore, in the following, the reader can suppose that $\mathcal{K}=\mathbb{R}^{d}$.
\end{rem}

\subsection{Entropy and symmetrization}

According to the Mock theorem \cite{lax1973hyperbolic,mock1980systems},
the system (\ref{eq:conslaw}) is symmetrizable and thus hyperbolic:
for all $W\in\mathcal{K}$ and all $N\in\mathbb{R}^{d}$ the Jacobian
of the flux
\[
A(W,N)=D_{W}Q(W,N)
\]
is diagonalizable with real eigenvalues.

Let us recall how to prove this result, because it will be useful.
For this we introduce the conjugate of the entropy \cite{hiriart2004fundamentals,hiriart2013convex}
defined by
\begin{equation}
s^{*}(W^{*})=\max_{W}(W^{*}\cdot W-s(W)),\label{eq:fenchel_transfo}
\end{equation}
where we denote by $\cdot$ the usual dot product: 
\[
W^{*}\cdot W=W^{*}{}^{\intercal}W=W{}^{\intercal}W^{*}.
\]
Thanks to (\ref{eq:infty_outside}) we can also write
\begin{equation}
s^{*}(W^{*})=\max_{W\in\mathcal{K}}(W^{*}\cdot W-s(W)),\label{eq:fenchel_transfo_K}
\end{equation}
The components of $W^{*}$ are called the dual variables, or entropy
variables \cite{harten1983symmetric,bourdel1989approximation,croisille1990contribution}.
The function $s^{*}$ is called the conjugate or the dual entropy.
The definition of the conjugate (\ref{eq:fenchel_transfo}) applies
to functions that are not necessarily smooth or convex. In the regular
case, when $s$ is smooth and strictly convex on $\mathbb{R}^{m}$,
for instance, $s^{*}$ is defined implicitly by the following two
relations
\begin{equation}
W^{*}=\nabla s(W(W^{*})),\label{eq:dualvar}
\end{equation}
\begin{equation}
s^{*}(W^{*})=W^{*}\cdot W(W^{*})-s(W(W^{*})).\label{eq:legendre_transfo}
\end{equation}
The formula (\ref{eq:dualvar}) determines uniquely the map
between the conservative variables $W$ and the dual variables $W^{*}$.
In the convex case, it can also be shown that $s^{**}=s$. Thus, we
have the reverse relations:
\[
W=\nabla s^{*}(W^{*}(W)),
\]
\[
s(W)=W^{*}(W)\cdot W-s^{*}(W^{*}(W)).
\]
We can also define the dual entropy flux by the relation
\begin{equation}
P^{i,\star}(W^{*})=W^{*}\cdot Q^{i}(W(W^{*}))-g^{i}(W(W^{*})).\label{eq:dualflux}
\end{equation}
We do not use exactly the same symbol for the dual entropy ($*$)
and the dual flux ($\star$), because the definitions are slightly
different. An important fact is that the knowledge of $s^{*}(W^{*})$
and $P^{i,\star}(W^{*})$ is sufficient to reconstruct the system
of conservation laws (\ref{eq:conslaw}). Indeed,
\begin{align*}
\nabla_{W^{*}}P^{i,\star}(W^{*})&=Q^{i}(W(W^{*}))+W^{*}\cdot D_{W}Q^{i}(W(W^{*}))D_{W^{*}}W(W^{*})-D_{W}g^{i}D_{W^{*}}W(W^{*}),
\\
&=Q^{i}(W(W^{*}))+D_{W}s(W)D_{W}Q^{i}(W(W^{*}))D_{W^{*}}W(W^{*})-D_{W}g^{i}D_{W^{*}}W(W^{*}),
\\
\intertext{(because of (\ref{eq:transpose}) and (\ref{eq:entropyfluxrelation}))}
&=Q^{i}(W(W^{*}))+D_{W}g^{i}(W(W^{*}))D_{W^{*}}W(W^{*})-D_{W}g^{i}(W(W^{*}))D_{W^{*}}W(W^{*}),
\end{align*}
(because of (\ref{eq:dualvar})), and thus

\begin{equation}
\nabla_{W^{*}}P^{i,\star}(W^{*})=Q^{i}(W(W^{*})).\label{eq:dualfluxrel}
\end{equation}
In short: the gradient of the dual entropy gives the conservative
variables and the gradient of the dual flux gives the flux.
\begin{thm}
The change of variables $W^{*}\mapsto W(W^{*})$ symmetrizes the system
of conservation laws (\ref{eq:conslaw}).
\end{thm}

\begin{proof}
Indeed
\[
\partial_{t}W+\sum_{i}\partial_{i}Q^{i}(W)=0,
\]
can be rewritten
\[
\partial_{t}\nabla_{W^{*}}s^{*}(W^{*})+\sum_{i}\partial_{i}\nabla_{W^{*}}P^{i,\star}(W^{*})=0,
\]
thanks to the above remark (\ref{eq:dualvar}) and (\ref{eq:dualfluxrel}),
or
\begin{equation}
D_{W^{*}W^{*}}^{2}s^{*}(W^{*})\partial_{t}W^{*}+\sum_{i}D_{W^{*}W^{*}}^{2}P^{i,\star}(W^{*})\partial_{i}W^{*}=0,\label{eq:mock_form}
\end{equation}
where the Hessian matrices $D_{W^{*}W^{*}}^{2}s^{*}(W^{*})$ and $D_{W^{*}W^{*}}^{2}P^{i,\star}(W^{*})$
are obviously symmetric and $D_{W^{*}W^{*}}^{2}s^{*}(W^{*})$ is positive
definite when $s^{*}$ is strictly convex.
\end{proof}
We have just proven the Mock theorem \cite{mock1980systems}. In the
following sections, we shall often try to guess directly a symmetrization
of the system of conservation laws, rather than the full dual entropy
theory to obtain the stability conditions of the relaxation scheme.
\begin{rem}
In the following, the practical stability will derive from a convexity
condition imposed on $s^{*}$. This may sound strange, because the
conjugate of any function is always convex, according to the standard
definition (\ref{eq:fenchel_transfo}). We need to be a little bit
more subtle. Actually, it is possible to define two different duality
transformations. Let us denote by Fenchel transform the $s^{*}$ given
by the first formula (\ref{eq:fenchel_transfo}). The Fenchel transform
derives from an optimization problem. Let us denote by Legendre transformation
the $s^{*}$ given by formula (\ref{eq:dualvar}) and (\ref{eq:legendre_transfo}).
The Legendre transform is thus defined by purely algebraic relations.
It is unambiguously defined, as soon as $W\mapsto W^{*}=\nabla_{W}s(W)$
is an invertible map. In this algebraic definition, $s^{*}$ is not
necessarily convex, but the convexity of $s^{*}$ is equivalent to
the convexity of $s$. The Legendre and Fenchel transform are different
in the general case. They coincide in the convex case.
\end{rem}

$\,$
\begin{rem}
In the most general case, the Legendre transform is a multivalued
function. It has to be defined from differential geometry tools. The
interested reader can refer to \cite{ekeland1977legendre} for an
introduction to this topic.
\end{rem}

\section{Vectorial kinetic model\label{sec:vec_kin}}

\subsection{Direct construction}

In this section, we recall the entropy theory of the kinetic representation.
This theory has a long history, see for instance \cite{lax1971shock,mock1980systems,harten1983symmetric,deshpande1986second,bourdel1989approximation,croisille1990contribution,perthame1990boltzmann}.
In our context it has been analyzed by Bouchut in \cite{bouchut1999construction,bouchut2003entropy}.
The ideas have been rephrased \modifb{with a systematic use of the Legendre transform} in
the work of Dubois \cite{dubois2014entropy}. Let us now recall the theory.

We consider a formal vectorial kinetic representation of system (\ref{eq:conslaw})
\begin{equation}
\partial_{t}F_{k}+V_{k}\cdot\nabla F_{k}=\frac{1}{\varepsilon}(F_{k}^{\text{\text{eq}}}-F_{k}),\quad k=1\ldots n_{v}.\label{eq:kin_relax}
\end{equation}
The approximate conservative vector is the sum of the kinetic vectors
\begin{equation}
W=\sum_{k=1}^{n_v} F_{k},\label{eq:sum_kin_eq_w}
\end{equation}
and the kinetic equilibrium vectors (or Maxwellians) are functions
of the conservative data 
\begin{equation}
F_{k}^{\text{\text{eq}}}=F_{k}^{\text{\text{eq}}}(W).\label{eq:def_equilibre}
\end{equation}
The kinetic velocities $V_{k}$ are $n_{v}$ constant and given vectors
of $\mathbb{R}^{d}$. In practice, it is interesting to introduce
a positive parameter $\lambda$, whose purpose is to change the size
of the kinetic velocities. So we shall often take
\begin{equation}
V_{k}=\lambda\tilde{V}_{k},\label{eq:scaling}
\end{equation}
where the directions $\tilde{V}_{k}$ of the kinetic velocities are
fixed. In this way, the kinetic velocities can be dilated.

For practical reasons we also introduce the kinetic vectors $F$ and
$F^{eq}$, which are column vectors made of all the stacked kinetic
data:
\[
F=(F_{1}^{\intercal},\ldots,F_{n_{v}}^{\intercal})^{\intercal},\quad F^{eq}=\left(\left(F_{1}^{eq}\right)^{\intercal},\ldots,\left(F_{n_{v}}^{eq}\right)^{\intercal}\right)^{\intercal}.
\]
With the help of the following diagonal matrices ($1_{m}$ is the
identity matrix of size $m\times m$)
\[
V^{i}=\left(\begin{array}{ccc}
V_{1}^{i}1_{m}\\
 & \ddots\\
 &  & V_{n_{v}}^{i}1_{m}
\end{array}\right),
\]
the kinetic system can also be written in the compact form
\[
\partial_{t}F+\sum_{i=1}^{d}\partial_{i}\left(V^{i}F\right)=\frac{1}{\varepsilon}(F^{\text{\text{eq}}}-F).
\]
When $\varepsilon\to0$, we expect that
\begin{equation}
F_{k}\simeq F_{k}^{eq}.\label{eq:expect_approx}
\end{equation}
If we assume that
\begin{equation}
W=\sum_{k=1}^{n_v} F_{k}^{\text{\text{eq}}}(W),\quad Q^{i}(W)=\sum_{k=1}^{n_{v}}V_{k}^{i}F_{k}^{\text{\text{eq}}}(W),\label{eq:algebraic_consistence}
\end{equation}
then, summing (\ref{eq:kin_relax}) on $k$ and using (\ref{eq:expect_approx})
we formally obtain
\[
\partial_{t}W+\sum_{i=1}^{d} \partial_{i}Q^{i}(W)\simeq0,
\]
and we have obtained an approximation of the initial system of conservation
laws.
\begin{rem}
For being more rigorous, we should have made more explicit in (\ref{eq:kin_relax})
and (\ref{eq:sum_kin_eq_w}) the dependency of $F_{k}$ and $W=\sum_{k}F_{k}$
with respect to $\varepsilon$. For instance, by using the notations
$F_{k}^{(\varepsilon)}$ and $W^{(\varepsilon)}$. But we prefer to
lighten the notations, and in the following the dependency with $\varepsilon$
will be implicit. This means that from now on, $W=\sum_{k}F_{k}$
is not an exact solution of (\ref{eq:conslaw}) but an approximate
one.
\end{rem}

For the moment, instead of relating the equilibrium (\ref{eq:def_equilibre})
to the conservation laws (\ref{eq:conslaw}), we assume that the equilibrium
is obtained from an entropy optimization principle. For this we introduce
a microscopic entropy
\[
\sigma(F)=\sum_{k=1}^{n_v} s_{k} (F_{k}),
\]
where the kinetic entropies $s_{k}$ are strictly convex functions
of the $F_{k}$ on $\mathcal{K}$. The macroscopic entropy is obtained
from the resolution of the following constrained optimization problem
\begin{equation}
s(W)=\min_{W=\sum_{k}F_{k}}\sigma(F).\label{eq:minprob}
\end{equation}
In optimization theory, this operation is known as an inf-convolution
operation \cite{hiriart2013convex}. The macroscopic entropy is the
inf-convolution of the kinetic entropies. In many works, the inf-convolution
operator is denoted with a $\square$. We thus have:
\[
s=s_{1}\square s_{2}\ldots\square s_{n_{v}}.
\]
We denote by $F_{k}^{\text{\text{eq}}}(W)$ the (supposed to be unique)
values of $F_{k}$ that achieve the minimum
\[
s(W)=\sum_{k=1}^{n_v} s_{k}(F_{k}^{\text{\text{eq}}}(W)).
\]
If we introduce the Lagrangian
\[
L(F,\Lambda)=\sum_{k=1}^{n_v}s_{k}(F_{k})+\Lambda\cdot \left(W-\sum_{k=1}^{n_v}F_{k}\right).
\]
The minimizer $F^{\text{\text{eq}}}(W)$ and the Lagrange multiplier
$\Lambda(W)$ are characterized by
\begin{equation}
\nabla_{F_{k}}s_{k}(F_{k}^{\text{\text{eq}}})=\Lambda,\quad\sum_{k=1}^{n_v}F_{k}^{\text{\text{eq}}}=W.\label{eq:lagcond}
\end{equation}
These relations are simply obtained by deriving the Lagrangian with
respect to $F_{k}$ or $\Lambda$.

An essential property of the inf-convolution is that the Fenchel transform
changes it into a sum. We thus have
\begin{equation}
s^{*}(W^{*})=\sum_{k=1}^{n_{v}}s_{k}^{*}(W^{*}).\label{eq:entropy_sum}
\end{equation}
Taking the Legendre transform of (\ref{eq:lagcond}) we see that the
couple $(F_{k}^{\text{\text{eq}}}(W),\Lambda(W))$ is  a solution to
\begin{equation}
F_{k}^{\text{\text{eq}}}=\nabla_{\Lambda}s_{k}^{*}(\Lambda(W)),\quad\sum_{k=1}^{n_v}F_{k}^{\text{\text{eq}}}=W.\label{eq:microdual}
\end{equation}
Summing over $k$ we also have the Lagrange multiplier in an easier
way
\[
\sum_{k=1}^{n_v}F_{k}^{\text{\text{eq}}}=\sum_{k=1}^{n_v}\nabla_{\Lambda}s_{k}^{*}(\Lambda(W)),
\]
\[
W=\nabla_{\Lambda}\sum_{k=1}^{n_v}s_{k}^{*}(\Lambda(W)),
\]
\[
=\nabla_{\Lambda}s^{*}(\Lambda(W)),
\]
and thus
\[
\Lambda(W)=W^{*}(W)=\nabla s(W).
\]
The Lagrange multiplier of the constrained optimization problem (\ref{eq:lagcond})
is simply the gradient of the macroscopic entropy.

Let us now assume the additional property
\begin{equation}
P^{i,\star}(W^{*})=\sum_{k=1}^{n_{v}}V_{k}^{i}s_{k}^{*}(W^{*}).\label{eq:maindual}
\end{equation}
Then we have
\begin{align*}
\nabla_{W^{*}}P^{i,\star}(W^{*})&=\sum_{k=1}^{n_{v}}V_{k}^{i}\nabla_{W^{*}}s_{k}^{*}(W^{*}),
\\
\nabla_{W^{*}}P^{i,\star}(W^{*})&=\sum_{k=1}^{n_{v}}V_{k}^{i}F_{k}^{\text{\text{eq}}}(W),
\end{align*}
(from (\ref{eq:microdual})) and thus, from (\ref{eq:dualfluxrel})
\begin{equation}
Q^{i}(W)=\sum_{k=1}^{n_{v}}V_{k}^{i}F_{k}^{\text{\text{eq}}}(W).\label{eq:consistency}
\end{equation}
In the flux form, we can also write
\[
Q(W,N)=\sum_{k=1}^{n_{v}}(V_{k}\cdot N)F_{k}^{\text{\text{eq}}}(W)
.\]
We recover the relations (\ref{eq:algebraic_consistence}) that impose
the consistency of the kinetic model (\ref{eq:kin_relax}) with the
system of conservation laws (\ref{eq:conslaw}). But now the consistency
derives from an entropy optimization principle. We can sum up the
direct construction of a kinetic model (\ref{eq:kin_relax}) that
is entropy consistent with (\ref{eq:conslaw}):
\begin{enumerate}
\item Compute the Legendre transform $s^{*}$ of the entropy $s$ by (\ref{eq:dualvar}),
(\ref{eq:legendre_transfo}).
\item Compute the dual fluxes $P^{i,\star}$ by (\ref{eq:dualflux}).
\item Find $n_{v}$ strictly convex functions $s_{k}^{*}$ satisfying the
consistency relations written in the dual variables: 
\begin{equation}
\sum_{k=1}^{n_v}s_{k}^{*}=s^{*},\quad\sum_{k=1}^{n_v}V_{k}^{i}s_{k}^{*}=P^{i,\star}.\label{eq:dual_consistency}
\end{equation}
\item The equilibrium is given by
\[
F_{k}^{eq}=\nabla s_{k}^{*}.
\]
\end{enumerate}
\begin{rem}
If $n_{v}\geq d+1$, there is generally at least one solution to the
algebraic system (\ref{eq:dual_consistency}). The difficulty is to
ensure that the $s_{k}^{*}$ are convex. For this property to hold,
the scaling (\ref{eq:scaling}) is useful. It gives an additional
degree of freedom for ensuring convexity.
\end{rem}
\modifa{
\begin{rem}
The main advantage of constructing the kinetic entropy by the dual analysis is theoretical.
In the end the algebraic consistency relations given in (\ref{eq:algebraic_consistence}) are recovered.
The dual analysis ensures that the equilibrium kinetic vector $F_k^{eq}$ corresponds to a minimum of the 
kinetic entropy.
\end{rem}
}

\subsection{Reverse construction}

Now that we have recalled the entropy theory of the kinetic representation,
we can proceed in the reverse way. We \textsl{choose} the equilibrium
$F^{\text{\text{eq}}}$ in such a way that the consistency relation
(\ref{eq:algebraic_consistence}) is satisfied. From the above theory,
we expect that $F_{k}^{\text{\text{eq}}}(W(W^{*}))$ is a gradient,
when it is expressed in the entropy variables $W^{*}$. We can thus
find dual kinetic entropies $s_{k}^{*}(W^{*})$ such that
\[
F_{k}^{\text{\text{eq}}}(W(W^{*}))=\nabla_{W^{*}}s_{k}^{*}(W^{*}).
\]
By Legendre transform, we can (in principle) compute the kinetic entropies
$s_{k}(F_{k})$ and this gives us the microscopic entropy
\[
\sigma(F)=\sum_{k=1}^{n_v}s_{k}(F_{k}).
\]
The main point in the reverse construction is to ensure that the strict
convexity is preserved. In practice, we will see that the microscopic
entropy is convex under a condition that $\lambda$ is large enough.
This will be the subcharacteristic condition. 

\section{Splitting and over-relaxation}

\subsection{Splitting scheme}

In this section we recall how to solve practically in time the kinetic
system (\ref{eq:kin_relax}). Indeed, in (\ref{eq:kin_relax}) all
the transport equations are coupled in a non-linear way. We introduce
a splitting method for separating the transport equations. In this
numerical method, the approximation of the kinetic data is computed
at fixed times $t_{n}=n\Delta t$, where $\Delta t$ is the time step.
The approximation is not continuous at time $t_{n}$, i.e. we distinguish
between the values of $F_{k}(X,t_{n}^{-})$ and of $F_{k}(X,t_{n}^{+})$.
Suppose that we know the kinetic data $F_{k}(X,t_{n-1}^{+})$ at the
end of time step $n-1$. Computing the next time step consists first
in solving the homogenous transport equations
\[
\partial_{t}F_{k}+V_{k}\cdot\nabla F_{k}=0,
\]
for a duration of $\Delta t$, which defines $F_{k}(X,t_{n}^{-}).$
Indeed, using the characteristic method (and avoiding for the moment
difficulties arising from the boundaries), we get the explicit formula
\begin{equation}
F_{k}(X,t_{n}^{-})=F_{k}(X-\Delta tV_{k},t_{n-1}^{+}).\label{eq:characteristic}
\end{equation}
In this way, we obtain the new conservative data by
\[
W(X,t_{n})=\sum_{k=1}^{n_v} F_{k}(X,t_{n}^{-}).
\]
We then apply the following over-relaxation formula
\begin{equation}
F_{k}(X,t_{n}^{+})=\omega F_{k}^{eq}(W(X,t_{n}))+(1-\omega)F_{k}(X,t_{n}^{-}).\label{eq:over_relax}
\end{equation}
Several choices can be made for the relaxation parameter $\omega$.
The choice $\omega=1$ corresponds to a projection of the kinetic
data on the equilibrium at the end of each time step. The choice $\omega=2$
(over-relaxation) is interesting, because, in this case, it can be
shown that the time integration is second-order accurate (see for
instance \cite{coulette2019high} and included references).

\subsection{Operator notations\label{subsec:op_notations}}

The splitting scheme is very simple to implement in a computer program
and no additional information is needed. However, we introduce here
some additional mathematical notations that will be useful for deriving
the equivalent equation analysis.

The interesting output of the kinetic scheme is obviously the conservative
variables vector of size $m$, given by
\[
W=\sum_{k=1}^{n_{v}}F_{k}.
\]
However, the kinetic representation introduces $n_{v}$ kinetic vectors
$F_{k}$ instead of one vector $W$. For the analysis, it is necessary
to supplement $W$ with additional quantities. The equivalent equation
could be written in an arbitrary set of variables $Y\in\mathbb{R}^{mn_{v}}$.
We have done a particular choice, which helps the understanding, to
our opinion.

The first components of $Y$ are made of the conservative variables
$W$. The next components will be made of the flux errors
\[
\sum_{k=1}^{n_v}V_{k}^{i}\left(F_{k}-F_{k}^{eq}\right)=\sum_{k=1}^{n_v}V_{k}^{i}F_{k}-Q^{i},\quad1\leq i\leq d.
\]
If $n_{v}=d+1$, we have enough variables. If $n_{v}>d+1$ we still
have to supplement $Y$ with $n_{c}=(n_{v}-d-1)$ independent linear
combinations of the kinetic data
\[
\sum_{k=1}^{n_v}\beta_{k}^{\ell}\left(F_{k}-F_{k}^{eq}\right),\quad1\leq\ell\leq n_{v}-d-1=n_{c}.
\]
It is convenient to choose the coefficients $\beta_{k}^{\ell}$ of
the linear combinations in order to cancel the contributions of $F^{eq}$
\begin{equation}
\sum_{k=1}^{n_v}\beta_{k}^{\ell}F_{k}^{eq}=0.\label{eq:cancel_feq}
\end{equation}
Finally, $Y$ is of the form
\[
Y=\left(\begin{array}{c}
W\\
\sum_{k}V_{k}^{1}F_{k}\\
\vdots\\
\sum_{k}V_{k}^{d}F_{k}\\
\sum_{k}\beta_{k}^{1}F_{k}\\
\vdots\\
\sum_{k}\beta_{k}^{n_{v}-d-1}F_{k}
\end{array}\right)-\left(\begin{array}{c}
0\\
\sum_{k}V_{k}^{1}F_{k}^{eq}\\
\vdots\\
\sum_{k}V_{k}^{d}F_{k}^{eq}\\
0\\
\vdots\\
0
\end{array}\right).
\]
While the $m$ first components of $Y$ are an approximate solution
to (\ref{eq:conslaw})
\[
Y^{1\cdots m}=W,
\]
we expect the next components to be small 
\[
Y^{m+1\cdots mn_{v}}\simeq0.
\]
In the following, we denote the sub-vector $Y^{m+1\cdots mn_{v}}$
as the \textsl{flux error} (even if the last components of $Y$, as
we have seen above, are not necessarily all related to fluxes).

We can write the $Y$ representation in a matrix form
\begin{equation}
Y(F)=MF-CMF^{eq}(W(F)),\quad W(F)=\sum_{k=1}^{n_v}F_{k}.\label{eq:F_to_Y}
\end{equation}
We call the matrix $M$ the matrix of moments. It is supposed to be
invertible. The matrix $C$ is the diagonal matrix
\[
C=\left(\begin{array}{ccc}
0_{m}\\
 & 1_{md}\\
 &  & 0_{mn_{c}}
\end{array}\right),
\]
where we denote by $1_{r}$ the diagonal identity matrix of size $r\times r$
and by $0_{r}$ the null matrix of size $r\times r$. The nonlinear
mapping $F\mapsto Y$ can be inverted explicitly
\begin{equation}
F(Y)=M^{-1}Y+M^{-1}CMF^{eq}(W(Y)),\quad W(Y)=(Y{{}^1},\ldots,Y^{m})^{\intercal}.\label{eq:Y_to_F}
\end{equation}
The notations are tedious but simple. For making them clearer, let
us consider three examples that we will use below.

\subsubsection{D1Q2 case}

In this case, the space dimension $d=1$ and we take two kinetic velocities
\[
V_{1}=-\lambda,\quad V_{2}=\lambda.
\]
Because $n_{v}=d+1$ and from the consistency relation (\ref{eq:consistency}),
the equilibrium is necessarily given by
\[
F_{1}^{eq}(W)=\frac{1}{2}W-\frac{1}{2\lambda}Q(W),\quad F_{2}^{eq}(W)=\frac{1}{2}W+\frac{1}{2\lambda}Q(W).
\]
Then 
\[
M=\left(\begin{array}{cc}
1_{m} & 1_{m}\\
-\lambda1_{m} & \lambda1_{m}
\end{array}\right).
\]

\subsubsection{D2Q3 case}

In this case, the space dimension $d=2$ and we take $n_{v}=3$ kinetic
velocities
\[
V_{1}=\lambda\left(\begin{array}{c}
1\\
0
\end{array}\right),\quad V_{2}=\frac{\lambda}{2}\left(\begin{array}{c}
-1\\
\sqrt{3}
\end{array}\right),\quad V_{3}=\frac{\lambda}{2}\left(\begin{array}{c}
-1\\
-\sqrt{3}
\end{array}\right).
\]
Because of $n_{v}=d+1$ the equilibrium is uniquely defined. It is
given by
\[
F^{eq}=M^{-1}\left(\begin{array}{c}
W\\
Q^{1}\\
Q^{2}
\end{array}\right),
\]
with 
\[
M=\left(\begin{array}{ccc}
1_{m} & 1_{m} & 1_{m}\\
\lambda1_{m} & -\frac{1}{2}\lambda1_{m} & -\frac{\lambda}{2}1_{m}\\
0_{m} & \frac{\sqrt{3}}{2}\lambda1_{m} & -\frac{\sqrt{3}}{2}\lambda1_{m}
\end{array}\right).
\]
We find
\[
F_{k}^{eq}=\frac{W}{3}+\frac{2}{3\lambda^{2}}V_{k}\cdot\left(\begin{array}{c}
Q^{1}\\
Q^{2}
\end{array}\right).
\]

\subsubsection{D2Q4 case}

In this case, the space dimension $d=2$ and we take four kinetic
velocities
\[
V_{1}=\lambda\left(\begin{array}{c}
1\\
0
\end{array}\right),\quad V_{2}=\lambda\left(\begin{array}{c}
-1\\
0
\end{array}\right),\quad V_{3}=\lambda\left(\begin{array}{c}
0\\
1
\end{array}\right),\quad V_{4}=\lambda\left(\begin{array}{c}
0\\
-1
\end{array}\right).
\]
Because $n_{v}>d+1$, we have one degree of freedom for the computation
of the equilibrium. We can take, by analogy with the D1Q2 model:
\[
F_{1,2}^{eq}(W)=\modifa{\frac{1}{4}}W\pm\frac{1}{2\lambda}Q^{1}(W),\quad F_{3,4}^{eq}(W)=\modifa{\frac{1}{4}}W\pm\frac{1}{2\lambda}Q^{2}(W).
\]
Then the matrix of moments is given by 
\[
M=\left(\begin{array}{cccc}
1_{m} & 1_{m} & 1_{m} & 1_{m}\\
\lambda1_{m} & -\lambda1_{m} & 0_{m} & 0_{m}\\
0_{m} & 0_{m} & \lambda1_{m} & -\lambda1_{m}\\
\lambda^{2}1_{m} & \lambda^{2}1_{m} & -\lambda^{2}1_{m} & -\lambda^{2}1_{m}
\end{array}\right).
\]
For the D2Q4 model, other choices are possible for the last row of
the matrix of moments. We have made the choice proposed in \cite{fevrier2014extension},
but we could have replaced $\lambda^{2}$ by one, for instance. Actually,
we could have supplemented by any row that ensures that $M$ is invertible,
but then the condition (\ref{eq:cancel_feq}) is generally not satisfied.

\subsubsection{Transport operator in the $Y$ variables}

With these notations, we can rewrite the transport operator in the
$Y$ variables. Let us define the $\tau_{k}(\Delta t)$ shift operators,
applied on an arbitrary function $X\mapsto f(X)$ by
\[
(\tau_{k}(\Delta t)f)(X)=f(X-\Delta tV_{k}).
\]
The vectorial shift operator is then given by
\[
\tau(\Delta t)=\left(\begin{array}{ccc}
\tau_{1}(\Delta t)1_{m} &  & 0\\
 & \ddots\\
0 &  & \tau_{n_{v}}(\Delta t)1_{m}
\end{array}\right).
\]
The initial data $Y(\cdot,t_{n-1}^{+})$ at the end of the time step $n-1$
are supposed to be known. The transport map $\mathcal{T}(\Delta t)$
is the procedure that takes $Y(\cdot,t_{n-1}^{+})$ and produces $Y(\cdot,t_{n}^{-})$
before the relaxation step:
\[
Y(\cdot,t_{n}^{-})=\mathcal{T}(\Delta t)Y(\cdot,t_{n-1}^{+}).
\]
The procedure is as follows:
\begin{enumerate}
\item Express the initial data $Y$ in the kinetic variables thanks to (\ref{eq:Y_to_F})
\[
F(\cdot,t_{n-1}^{+})=F(Y(\cdot,t_{n-1}^{+})).
\]
\item Apply the transport operator
\[
F(\cdot,t_{n}^{-})=\tau(\Delta t)F(\cdot,t_{n-1}^{+}).
\]
\item Go back to the $Y$ variables thanks to (\ref{eq:F_to_Y})
\[
Y(\cdot,t_{n}^{-})=Y(F(\cdot,t_{n}^{-})).
\]
\end{enumerate}
The transport is nonlinear functional operator, made of shift operators
and nonlinear algebraic operations. We do not give all the details. 

\subsubsection{Relaxation operator in the $Y$ variables}

The relaxation operator in the kinetic variables is given by (\ref{eq:over_relax}).
In vectorial form it reads
\begin{equation}
F(\cdot,t_{n}^{+})=\omega F^{eq}(W(\cdot,t_{n}^{-}))+(1-\omega)F(\cdot,t_{n}^{-}).\label{eq:vec_over_relax}
\end{equation}
Using the fact that $W=\sum_{k}F_{k}$ does not change during the
relaxation, we simply obtain
\begin{equation}
Y(\cdot,t_{n}^{+})=\mathcal{R}_{\omega}Y(\cdot,t_{n}^{-}),\label{eq:relax_Y}
\end{equation}
where the relaxation operator $\mathcal{R}_{\omega}$ is the diagonal
operator
\[
\mathcal{R}_{\omega}=\left(\begin{array}{cc}
1_{m} & 0\\
0 & (1-\omega)1_{m(n_{v}-1)}
\end{array}\right).
\]
In other words, the relaxation operator leaves the conservative variables
unchanged:
\[
\left(\mathcal{R}_{\omega}Y\right)^{i}=Y^{i}=W^{i},\quad1\leq i\leq m,
\]
and the other components of $Y$ are multiplied by $1-\omega$:
\[
\left(\mathcal{R}_{\omega}Y\right)^{i}=(1-\omega)Y^{i},\quad i>m.
\]
When $\omega\simeq2$, $1-\omega\simeq-1$ and thus, these components
are numerically fluctuating around $0$ at each time iteration. There
is no damping of the oscillations when $\omega=2$. The frequency
of the fluctuations is $1/\Delta t$ and tends to infinity when $\Delta t$
tends to zero.

\subsubsection{Formal numerical scheme}

Now that we have introduced the operator notation, we can give a more
formal definition of the split kinetic approximation. For solving
(\ref{eq:conslaw}) with the initial data
\[
W(X,0)=W_{0}(X),
\]
we start with
\[
Y_{0}^{+}=\left(\begin{array}{c}
W_{0}(X)\\
0
\end{array}\right).
\]
Assume that we have reached $Y_{n-1}^{+}$ at time step $n-1$. Then
the next value is given by
\[
Y_{n}^{+}=\mathcal{S}(\Delta t)Y_{n-1}^{+},
\]
where the split  kinetic operator is 
\begin{equation}\label{eq:split_nsym}
\mathcal{S}(\Delta t)=\mathcal{R}_{\omega}\circ\mathcal{T}(\Delta t).
\end{equation}
Then we extract the first $m$ components of $Y_{n}^{+}$ and we expect
to obtain an approximation of $W$ at time $n$:
\[
Y_{n}^{+}(X)=\left(\begin{array}{c}
W_{n}(X)\\
\vdots
\end{array}\right),\quad W(X,t_{n})\simeq W_{n}(X).
\]
As stated above, the relaxation operator generates time fluctuations
with frequency $1/\Delta t$. For the analysis, it is convenient to
suppress these non-relevant time fluctuations. This is done by considering 
only an even number of time steps of (\ref{eq:split_nsym}) for instance.
\modifa{Another point is that the split operator (\ref{eq:split_nsym}) is not symmetric in time.
In order to remove the fast time oscillations and the lack of symmetry, we can consider
the modified split kinetic operator}
\begin{equation}
\mathcal{S}(\Delta t)=\mathcal{T}\left(\frac{\Delta t}{4}\right)\circ\mathcal{R}_{\omega}\circ\mathcal{T}\left(\frac{\Delta t}{2}\right)\circ\mathcal{R}_{\omega}\circ\mathcal{T}\left(\frac{\Delta t}{4}\right).\label{eq:sym_split}
\end{equation}
In the case $\omega=2$, this modified operator is indeed time symmetric
in the sense that
\begin{equation}
\mathcal{S}(0)=I_{d},\quad\mathcal{S}(\Delta t)^{-1}=\mathcal{S}(-\Delta t).\label{eq:sym_prop}
\end{equation}
This is this property that ensures that the split scheme is second
order accurate when $\omega=2$. We refer, for instance to \cite{coulette2019high}
and to the consistency analysis given below.
\modifa{The symmetric operator and the non-symmetric operators sampled at the even time steps are similar.
They differ in the initial and final steps by a term of order $\Delta t$. This difference has in practice very few effect 
on the order of the scheme on the conservative variables (see Section 6.1.1 in \cite{coulette2019high}).}

\section{Applications}

In this section, we apply the entropy theory to practical examples.

\subsection{Application to the transport equation \label{subsec:trans_example}}

We first consider the construction of the D1Q2 model for the one-dimensional
transport equation
\[
\partial_{t}W+c\partial_{x}W=0,
\]
where the velocity $c$ is supposed to be constant and $X=x$ is the
one-dimensional space variable. In this case
\[
Q(W)=cW,\quad V_{1}=-\lambda,\quad V_{2}=\lambda,
\]
and we have no free choice for choosing the equilibrium kinetic data,
which are given by
\[
F_{1}^{\text{\text{eq}}}(W)=\frac{W}{2}-\frac{cW}{2\lambda},\quad F_{2}^{\text{\text{eq}}}(W)=\frac{W}{2}+\frac{cW}{2\lambda}.
\]
For this simple linear conservation law we can take the entropy associated
to the $L^{2}$ norm
\[
s(W)=\frac{W^{2}}{2}.
\]
The dual entropy is simply
\[
s^{*}(W^{*})=\frac{W^{*}{}^{2}}{2},
\]
and the entropy variable is 
\[
W^{*}=\nabla_{W}s(W)=W.
\]
Thus 
\[
F_{1}^{\text{\text{eq}}}(W(W^{*}))=\frac{W^{*}}{2}-\frac{cW^{*}}{2\lambda},\quad F_{2}^{\text{\text{eq}}}(W(W^{*}))=\frac{W^{*}}{2}+\frac{cW^{*}}{2\lambda}.
\]
From (\ref{eq:microdual}) we deduce the dual kinetic entropies
\[
s_{1}^{*}(W^{*})=\frac{1}{4}(1-c/\lambda)W^{*}{}^{2},\quad s_{2}^{*}(W^{*})=\frac{1}{4}(1+c/\lambda)W^{*}{}^{2}.
\]
They are strictly convex under the subcharacteristic condition
\begin{equation}
\lambda>\left|c\right|.\label{eq:trans_sub}
\end{equation}
We can then compute the kinetic entropies
\[
s_{1}(F_{1})=\frac{\lambda}{\lambda-c}F_{1}^{2},\quad s_{2}(F_{2})=\frac{\lambda}{\lambda+c}F_{2}^{2}.
\]
The microscopic entropy is then
\[
\sigma(F_{1},F_{2})=\frac{\lambda}{\lambda-c}F_{1}^{2}+\frac{\lambda}{\lambda+c}F_{2}^{2}.
\]
As expected, it is a diagonal quadratic form in the $(F_{1},F_{2})$
variables.

Let us express the microscopic entropy with respect to the $Y=(W,y)^{\intercal}$
variables. We have
\[
W=F_{1}+F_{2},
\]
and 
\begin{align*}
y&=-\lambda F_{1}+\lambda F_{2}-Q(W),
\\
&=-\lambda F_{1}+\lambda F_{2}-c(F_{1}+F_{2}).
\end{align*}
After simple computations, we find that the microscopic entropy is
also
\begin{equation}
\widetilde{\sigma}(W,y)=\sigma(F_{1},F_{2})=\frac{W^{2}}{2}+\frac{y^{2}}{2(\lambda^{2}-c^{2})}.\label{eq:sigma_wy}
\end{equation}
It is a convex function of $W$ and $y$ under condition (\ref{eq:trans_sub}).
As expected, it is minimal when the flux error $y$ vanishes. In addition, in the relaxation step, the entropy is exactly conserved when $\omega=2$
because
\begin{equation}
\widetilde{\sigma}(W,(1-\omega)y)=\widetilde{\sigma}(W,-y)=\widetilde{\sigma}(W,y).\label{eq:entrop_sym_y}
\end{equation}

We are now in a position to prove the entropy stability of the over-relaxation
scheme when $1\leq\omega\leq2$.
\begin{thm}
With periodic boundary conditions, or in an infinite domain, the over-relaxation
scheme is entropy stable under the sub-characteristic condition (\ref{eq:trans_sub})
when $1\leq\omega\leq2$.
\end{thm}

\begin{proof}
It is sufficient to prove the decrease of the entropy
\[
\mathcal{S}(t)=\int_{x}\sigma(F_{1}(x,t),F_{2}(x,t))=\int_{x}s_{1}(F_{1})+s_{2}(F_{2}),
\]
for a single time step. In the transport step, one solves
\[
\partial_{t}F_{k}+V_{k}\partial_{x}F_{k}=0,
\]
and thus the microscopic entropies
\[
\overline{s}_{k}(t)=\int_{x}s_{k}(F_{k}),
\]
are separately conserved
\[
\overline{s}_{k}(t+\Delta t^{-})=\overline{s}_{k}(t^{+}).
\]
In the relaxation step, $W$ is not changed and 
\[
y(x,t+\Delta t^{+})=(1-\omega)y(x,t+\Delta t^{-}),
\]
because $\left|1-\omega\right|\leq1$ we see from the expression (\ref{eq:sigma_wy})
of the entropy in the $(W,y)$ variable that the microscopic entropy
decreases pointwise, at each $x$. Therefore,
\[
\mathcal{S}(t+\Delta t^{+})\leq\mathcal{S}(t+\Delta t^{-}).
\]
\end{proof}

\subsection{Application to the shallow water equations}

In order to show that the approach still works for non-linear systems
of conservation laws, we try now to apply the above method to the
shallow water model where the unknowns are the water height $h(x,t)$
and the velocity $u(x,t).$ It reads
\[
\partial_{t}W+\partial_{x}Q(W)=0,
\]
with
\[
W=\left(\begin{array}{c}
h\\
hu
\end{array}\right),\quad Q(W)=\left(\begin{array}{c}
hu\\
hu^{2}+gh^{2}/2
\end{array}\right).
\]
We define the primitive variables
\[
v=\left(\begin{array}{c}
h\\
u
\end{array}\right).
\]
 For smooth solutions, we also have 
\[
\partial_{t}v+B(v)\partial_{x}v=0,
\]
with 
\[
B(v)=\left(\begin{array}{cc}
u & h\\
g & u
\end{array}\right).
\]
Assume that the Lax entropy $s(W)=H(v)$ is expressed in the primitive
variables, and that the entropy flux $G(W)=R(v)$. Then we must have
\[
D_{v}H(v)B(v)=D_{v}R(v).
\]
Denoting the partial derivatives with indices we obtain
\[
\left(\begin{array}{cc}
H_{h} & H_{u}\end{array}\right)\left(\begin{array}{cc}
u & h\\
g & u
\end{array}\right)=\left(\begin{array}{cc}
R_{h} & R_{u}\end{array}\right).
\]
We search $H$ under the form
\[
H(h,u)=h\frac{u^{2}}{2}+e(h).
\]
 Because 
\[
H_{h}=\frac{u^{2}}{2}+e'(h),\quad H_{u}=hu,
\]
this gives
\[
\frac{u^{3}}{2}+ue'+ghu=R_{h},\quad\frac{3hu^{2}}{2}+he'=R_{u}.
\]
We take
\[
R=h\frac{u^{3}}{2}+ue+gu\frac{h^{2}}{2}.
\]
Then 
\[
R_{u}=\frac{3hu^{2}}{2}+e+g\frac{h^{2}}{2}=\frac{3hu^{2}}{2}+he'.
\]
$e(h)$ is then solution of the differential equation
\[
e-he'+gh^{2}/2=0.
\]
A solution is 
\[
e(h)=\frac{gh^{2}}{2}.
\]
In the end, we find 
\[
s(W)=h\frac{u^{2}}{2}+\frac{gh^{2}}{2},\quad G(W)=h\frac{u^{3}}{2}+ugh^{2}.
\]
This allows us to compute the entropy variables
\begin{equation}
W_{1}^{*}=gh-\frac{u^{2}}{2},\quad W_{2}^{*}=u,\label{eq:p_to_v}
\end{equation}
and the reverse change of variables
\[
h=\frac{2W_{1}^{*}+W_{2}^{*}{}^{2}}{2g},\quad u=W_{2}^{*}.
\]
The equilibrium kinetic vectors are given by
\[
F_{1}^{\text{\text{eq}}}=\frac{W}{2}-\frac{Q(W)}{2\lambda},\quad F_{2}^{\text{\text{eq}}}=\frac{W}{2}+\frac{Q(W)}{2\lambda}.
\]
After some calculations, we can express this equilibrium in the entropy
variables
\[
F_{1}^{\text{\text{eq}}}=\left[\begin{array}{cc}
\frac{\left(\mathit{W_{2}^{*}}^{2}+2\mathit{W_{1}^{*}}\right)\left(\lambda-\mathit{W_{2}^{*}}\right)}{4g\lambda} & -\frac{\left(\mathit{W_{2}^{*}}^{2}+2\mathit{W_{1}^{*}}\right)\left(-4\mathit{W_{2}^{*}}\lambda+5\mathit{W_{2}^{*}}^{2}+2\mathit{W_{1}^{*}}\right)}{16g\lambda}\end{array}\right]^{\intercal},
\]
\[
F_{2}^{\text{\text{eq}}}=\left[\begin{array}{cc}
\frac{\left(\mathit{W_{2}^{*}}^{2}+2\mathit{W_{1}^{*}}\right)\left(\lambda+\mathit{W_{2}^{*}}\right)}{4g\lambda} & \frac{\left(\mathit{W_{2}^{*}}^{2}+2\mathit{W_{1}^{*}}\right)\left(4\mathit{W_{2}^{*}}\lambda+5\mathit{W_{2}^{*}}^{2}+2\mathit{W_{1}^{*}}\right)}{16g\lambda}\end{array}\right]^{\intercal}.
\]
From the above theory, we know that 
\[
F_{k}^{\text{\text{eq}}}=\nabla_{W^{*}}s_{k}^{*},
\]
for some dual kinetic entropies $s_{k}^{*}$. This is indeed the case
and after more calculations we find
\[
s_{1}^{*}=\frac{\left(\lambda-\mathit{W_{2}^{*}}\right)\left(\mathit{W_{2}^{*}}^{2}+2\mathit{W_{1}^{*}}\right)^{2}}{16g\lambda},\quad s_{2}^{*}=\frac{\left(\mathit{W_{2}^{*}}^{2}+2\mathit{W_{1}^{*}}\right)^{2}\left(\lambda+\mathit{W_{2}^{*}}\right)}{16g\lambda}.
\]
It is then possible to compute the Hessians of $s_{k}^{*}$ and express
them in the $(h,u)$ variables with (\ref{eq:p_to_v}). We find
\[
D_{W^{*}W^{*}}s_{1}^{*}=\left[\begin{array}{cc}
\frac{\lambda-u}{2g\lambda} & \frac{-gh+\lambda u-u^{2}}{2g\lambda}\\
\frac{-gh+\lambda u-u^{2}}{2g\lambda} & \frac{\left(gh+u^{2}\right)\lambda-3hug-u^{3}}{2g\lambda}
\end{array}\right],
\]
\[
D_{W^{*}W^{*}}s_{2}^{*}=\left[\begin{array}{cc}
\frac{\lambda+u}{2g\lambda} & \frac{gh+\lambda u+u^{2}}{2g\lambda}\\
\frac{gh+\lambda u+u^{2}}{2g\lambda} & \frac{\left(gh+u^{2}\right)\lambda+3hug+u^{3}}{2g\lambda}
\end{array}\right].
\]
The two matrices are positive definite iff the first diagonal terms
and the determinants are positive. This is equivalent to
\[
\lambda>\left|u\right|,
\]
\[
(\lambda-u)^{2}-gh>0,\quad(\lambda+u)^{2}-gh>0,
\]
which is again equivalent to
\[
\lambda>\left|u\right|+\sqrt{gh}.
\]
This is the expected sub-characteristic condition. It is difficult
to go farther because the Legendre transforms $s_{1}$ and $s_{2}$
of $s_{1}^{*}$ and $s_{2}^{*}$ cannot be computed explicitly. However,
we can reproduce the stability property of the linear case. The microscopic
entropy is given by
\[
\sigma(F_{1},F_{2})=s_{1}(F_{1})+s_{2}(F_{2}).
\]
Using the relations
\[
W=F_{1}+F_{2},
\]
\[
y=-\lambda F_{1}+\lambda F_{2}-Q(F_{1}+F_{2}),
\]
we deduce
\[
F_{1}=\frac{W}{2}-\frac{Q(W)}{2\lambda}-\frac{y}{2\lambda},\quad F_{2}=\frac{W}{2}+\frac{Q(W)}{2\lambda}+\frac{y}{2\lambda},
\]
and the microscopic entropy can be expressed in function of $W$ and
$y$
\[
\widetilde{\sigma}(W,y)=s_{1}\left(\frac{W}{2}-\frac{Q(W)}{2\lambda}-\frac{y}{2\lambda}\right)+s_{2}\left(\frac{W}{2}+\frac{Q(W)}{2\lambda}+\frac{y}{2\lambda}\right).
\]
For a fixed $W$ the minimum is achieved for $y=0$, therefore the
macroscopic entropy is
\[
s(W)=\widetilde{\sigma}(W,0),
\]
and
\[
\nabla_{y}\widetilde{\sigma}(W,0)=0.
\]
Then, with a Taylor expansion near to $y=0$, we get
\[
\widetilde{\sigma}(W,y)=\widetilde{\sigma}(W,-y)+O(\left|y\right|^{3}).
\]
The relation (\ref{eq:entrop_sym_y}) thus still holds but with a
third-order term in $y$. This means that the relaxation scheme with
$\omega=2$ is entropy preserving up to third order in $y$. In principle,
it is also possible to construct a scheme that preserves exactly the
entropy in the non-linear case. It is sufficient to choose the relaxation
parameter $\omega=\omega(W,y)$ in such way that
\begin{equation}
\widetilde{\sigma}(W,(1-\omega(W,y))y)=\widetilde{\sigma}(W,y).\label{eq:same_entrop}
\end{equation}
In practice, this would not be very interesting, one would get
\[
\omega(W,y)\simeq2
\]
 and $\omega(W,y)$ would have to be computed numerically by first
computing $s_{1}$ and $s_{2}$numerically and then by solving (\ref{eq:same_entrop})
also numerically.

What is interesting, however, is that the reasoning ensures the existence
of a relaxation parameter $\omega(W,y)\simeq2$, such that the whole
scheme is entropy preserving. And if the scheme is run with a smaller
relaxation parameter, it is ensured to be entropy stable.
\modifa{Similar ideas for controlling entropy dissipation in the over-relaxation step have been developed by Karlin in the standard LBM for Navier-Stokes.
See \cite{karlin1999perfect} or \cite{hosseini2023entropic} for a recent review article. See also \cite{brownlee2007stability}}.

\subsection{Application to isothermal Euler equations}

Finally, we also apply the theory to the isothermal Euler model, which
reads
\[
\partial_{t}W+\partial_{x}Q(W)=0,
\]
with 
\[
W=\left(\begin{array}{c}
\rho\\
\rho u
\end{array}\right),\quad Q(W)=\left(\begin{array}{c}
\rho u\\
\rho u^{2}+c^{2}\rho
\end{array}\right).
\]
In primitive variables $v=(\rho,u)^{\intercal}$, the system reads
\[
\partial_{t}v+B(v)\partial_{x}v=0,
\]
with
\[
B(v)=\left(\begin{array}{cc}
u & \rho\\
\frac{c^{2}}{\rho} & u
\end{array}\right).
\]
First, let us find a Lax entropy $H(v)$ and an entropy flux $R(v)$.
We must have
\[
(H_{\rho},H_{u})\left(\begin{array}{cc}
u & \rho\\
\frac{c^{2}}{\rho} & u
\end{array}\right)=(R_{\rho},R_{u}).
\]
We search for $H(v)$ in the form
\[
H=\rho\frac{u^{2}}{2}+e(\rho).
\]
Then
\[
R_{\rho}=\frac{u^{3}}{2}+e'(\rho)u+c^{2}u,
\]
\[
R_{u}=\rho\frac{3u^{2}}{2}+e'(\rho)\rho.
\]
Integrating the first equation with respect to $\rho$ we get
\[
R=\rho\frac{u^{3}}{2}+e(\rho)u+\rho c^{2}u+C(\rho),
\]
(we can take $C(\rho)=0$). And deriving with respect to $u$ we get
\[
\rho\frac{3u^{2}}{2}+e'(\rho)\rho=\rho\frac{3u^{2}}{2}+e(\rho)+\rho c^{2}.
\]
Therefore $e(\rho)$ is a solution of the differential equation
\[
e'(\rho)\rho=e(\rho)+\rho c^{2}.
\]
 We can take
\[
e(\rho)=c^{2}\rho(\ln\rho-1).
\]
Finally
\[
s(W)=\rho\frac{u^{2}}{2}+c^{2}\rho(\ln\rho-1),\quad G(W)=u(s(W)+c^{2}\rho).
\]
The entropy variables are
\[
W_{1}^{*}=-\frac{u^{2}}{2}+c^{2}\ln\rho,\quad W_{2}^{*}=u.
\]
The reverse change of variables is
\[
\rho=W_{1}=\exp\left(\frac{2W_{1}^{*}+W_{2}^{*}{}^{2}}{2c^{2}}\right),\quad\rho u=W_{2}^{*}\exp\left(\frac{2W_{1}^{*}+W_{2}^{*}{}^{2}}{2c^{2}}\right).
\]
The equilibrium distribution is
\[
F_{1}^{\text{\text{eq}}}=\frac{W}{2}-\frac{Q(W)}{2\lambda}=\frac{1}{2}\left(\begin{array}{c}
W_{1}-W_{2}/\lambda\\
W_{2}-(W_{2}^{2}/W_{1}+c^{2}W_{1})/\lambda
\end{array}\right).
\]
\[
F_{2}^{\text{\text{eq}}}=\frac{W}{2}+\frac{Q(W)}{2\lambda}=\frac{1}{2}\left(\begin{array}{c}
W_{1}+W_{2}/\lambda\\
W_{2}+(W_{2}^{2}/W_{1}+c^{2}W_{1})/\lambda
\end{array}\right).
\]
In the dual variables we get
\[
\nabla s_{1}^{*}=\frac{\exp\left(\frac{2W_{1}^{*}+W_{2}^{*}{}^{2}}{2c^{2}}\right)}{2\lambda}\left(\begin{array}{c}
\lambda-W_{2}^{*}\\
-W_{2}^{*}{}^{2}+\lambda W_{2}^{*}-c^{2}
\end{array}\right),
\]
\[
\nabla s_{2}^{*}=\frac{\exp\left(\frac{2W_{1}^{*}+W_{2}^{*}{}^{2}}{2c^{2}}\right)}{2\lambda}\left(\begin{array}{c}
\lambda+W_{2}^{*}\\
W_{2}^{*}{}^{2}+\lambda W_{2}^{*}+c^{2}
\end{array}\right),
\]
and finally
\[
s_{1}^{*}=\frac{\exp\left(\frac{2W_{1}^{*}+W_{2}^{*}{}^{2}}{2c^{2}}\right)}{2\lambda}(\lambda-W_{2}^{*}),\quad s_{2}^{*}=\frac{\exp\left(\frac{2W_{1}^{*}+W_{2}^{*}{}^{2}}{2c^{2}}\right)}{2\lambda}(\lambda+W_{2}^{*}).
\]
With similar calculations as for the shallow water system we find
the following sub-characteristic condition
\[
\lambda>c+\left|u\right|.
\]
As for the shallow water system, it is difficult to compute explicitly
$s_{1}$ and $s_{2}$.

\section{Equivalent equation analysis}

The entropy analysis of the over-relaxation scheme ensures the stability
of the scheme as soon as the sub-characteristic condition is satisfied.
However, for the moment it is not obvious that the scheme provides
an approximation of the system (\ref{eq:conslaw}). Indeed, the consistency
of the kinetic approximation with the system of conservation laws
is ensured, as soon as 
\[
F(X,t)\simeq F^{eq}(W(X,t)),
\]
and the over-relaxation formula (\ref{eq:vec_over_relax}) enforces
$F=F^{eq}$ at the end of the time step only when $\omega=1$. 

The objective of the equivalent system analysis is to provide a consistency
theory in the case $1<\omega\leq2$. This consistency analysis is
based on two ingredients: a Taylor expansion followed by a Chapman-Enskog
analysis.

\subsection{Taylor expansion: equivalent system in \(Y\)}

The first ingredient is a Taylor expansion of
$$
\frac{\mathcal{S}(\Delta t)-\mathcal{S}^{-1}(\Delta t)}{2\Delta t}
$$
when $\Delta t$ tends to zero. The Taylor
expansion provides an equivalent system of Partial Differential Equations
(PDE) expressed on $Y$. This system involves the whole vector $Y$,
which contains both $W$ and the flux error. We denote it by the \textsl{equivalent
system} in the following. It takes the general form
\begin{equation}
\partial_{t}Y+\frac{r(\omega)}{\Delta t}\left(\begin{array}{c}
0\\
Y^{m+1\cdots}
\end{array}\right)+\sum_{1\leq i\leq d}A^{i}(Y,\omega)\partial_{i}Y-\Delta t\sum_{1\leq i,j\leq d}B^{i,j}(Y,\omega)\partial_{i,j}Y=O(\Delta t^{2}).
\label{eq:equiv_system_Y}
\end{equation}
In the case $\omega=2$, the formula is simpler and we shall find
that
\begin{equation}
r(2)=0,\quad B^{i,j}(Y,2)=0,\label{eq:om2_rB}
\end{equation}
and that the matrices $A^{i}$ are of the form
\begin{equation}
A^{i}(Y,2)=\left(\begin{array}{cc}
D_{W}Q^{i}(W) & 0_{m}\\
0_{m(n_{v}-1)} & \times
\end{array}\right).\label{eq:om2_A}
\end{equation}
The Taylor expansion is thus sufficient to get the second-order time
consistency of the split scheme with the initial system of conservation law
when $\omega=2$. This consistency holds even when the flux
error is large. In addition, up to third order terms, the evolution of $W$ is uncoupled from the
evolution of the flux error. This surprising result relies essentially on the
symmetry property (\ref{eq:sym_prop}), which ensures that the stiff
terms in $\Delta t^{-1}$ vanish in the Taylor expansion.

\subsection{Asymptotic analysis: equivalent equation in \(W\)}

In the case $\omega \neq 2$, the equivalent system contains stiff terms.
And the evolution of $W$ is no more uncoupled from the evolution of the
flux error.
In order to remove this coupling, the second step is to perform an additional
asymptotic analysis (similar to the Hilbert or Chapman-Enskog expansion),
with the assumption that the flux error is of order $O(\Delta t)$.
From (\ref{eq:equiv_system_Y}) we can deduce an algebraic relation between
the flux error $Y^{m+1\cdots}$ and the gradient of $W$
\[
Y^{m+1\cdots}=-\frac{\Delta t}{r(\omega)}T\sum_{1\leq i\leq d}A^{i}(Y,\omega)\partial_{i}\left(\begin{array}{c}
W\\
0
\end{array}\right)+O(\Delta t^{2}),\quad T=\left(\begin{array}{cc}
0_{m} & 0_{m(n_{v}-1)}\\
0_{m(n_{v}-1)} & 1_{m(n_{v}-1)}
\end{array}\right).
\]
Reinjecting this approximation in the first row of (\ref{eq:equiv_system_Y})
provides a simpler system involving only $W$. We denote it by the
\textsl{equivalent equation} in the following. It takes the form
\begin{equation}
\partial_{t}W+\sum_{1\leq i\leq d}\partial_{i}Q^{i}(W)-\Delta t\sum_{1\leq i,j\leq d}\partial_{i}\left(D^{i,j}(W,\omega)\partial_{j}W\right)=O(\Delta t^{2}).\label{eq:equiv_eq_W}
\end{equation}

\subsection{Stability conditions}

Once we have obtained the equivalent system (\ref{eq:equiv_system_Y})
and the equivalent equation (\ref{eq:equiv_eq_W}) we can study their
stability.

\subsubsection{Equivalent system}

From the theory developed in Section (\ref{sec:vec_kin}), we have
stability under the subcharacteristic condition. This condition ensures
dissipation of the kinetic entropy $\tilde{\sigma}(Y)=\tilde{\sigma}(W,Y^{m+1\cdots mn_{v}})=\sum_{k}s_{k}(F_{k}).$
We also know that
\[
\tilde{\sigma}(W,0)=s(W),
\]
corresponds to the minimum of $\tilde{\sigma}$ with respect to the
flux error $Y^{m+1\cdots mn_{v}}$. It is therefore expected that
the change of variables $Y\mapsto\nabla_{Y}\tilde{\sigma}(Y)$ symmetrizes
the equivalent system (\ref{eq:equiv_system_Y}), which is thus hyperbolic.
We have also seen that the computation of the kinetic entropies $s_{k}$
is not necessarily easy. That is why in the following we try to find
directly a symmetrization for the first order part of the equivalent
system, rather than computing the kinetic entropy. We introduce the
following definition.
\begin{defn}
\label{def:hyperbolic}We shall say that the equivalent system (\ref{eq:equiv_system_Y})
is \textsl{hyperbolic} iff we can find a matrix $P^{0}(Y)$, $0\leq i\leq d$,
such that $P^{0}(Y)$ is symmetric positive definite and $P^{0}(Y)A^{i}(Y,\omega)$
is symmetric for $1\leq i\leq d$. 
\end{defn}

\begin{rem}
Our definition of hyperbolicity is stronger than the usual one, which
only states that 
\[
\sum_{i=1}^{d}N_{i}A^{i}(Y,\omega)
\]
is diagonalizable with real eigenvalues for all directions $N$. We
expect that 
\[
P^{0}(Y)=\nabla^{2}\tilde{\sigma}(Y),
\]
at least in the case of the linear transport equation discussed in
Section \ref{subsec:trans_example}. But we don't know if the stability
condition of Definition \ref{def:hyperbolic} is equivalent to the
convexity of the dual kinetic entropies. Maybe that it is true. For
a discussion around these questions, we refer to \cite{bouchut2005stability}.
\end{rem}

\subsubsection{Equivalent equation}

For the equivalent equation, we already know that its first order
part is hyperbolic. The stability thus depends on the sign of the
second-order terms. This gives another stability criterion. For obtaining
this criterion, we multiply (\ref{eq:equiv_eq_W}) on the left by
$D_{W}s(W)$ and we integrate by part the second-order term. The entropy
is dissipated if
\[
\sum_{ij,k,\ell}\partial_{k,\ell}^{2}s\partial_{i}W^{k}D^{i,j}\partial_{j}W^{\ell}\geq0.
\]
We thus introduce the quadratic form acting on a second-order tensor
$\alpha_{i}^{k}$:
\[
\alpha\mapsto q(\alpha)=\sum_{i,j,k,\ell}\partial_{k,\ell}^{2}s(W)D^{i,j}(W,\omega)\alpha_{i}^{k}\alpha_{j}^{l}.
\]

\begin{defn}
The equivalent equation (\ref{eq:equiv_eq_W}) is \textsl{dissipative}
iff the quadratic form $q(\alpha)$ is positive. 
\end{defn}

\begin{rem}
This definition amounts to checking that the Hessian matrix $\nabla_{\alpha}^{2}q(\alpha)$
is positive. This (symmetric) matrix is of size $md\times md$. In
the scalar case $m=1$, which we study below, the condition is simpler.
It simply states that the quadratic form
\[
x\mapsto\sum_{i,j}D^{i,j}x_{i}x_{j}
\]
is positive.
\end{rem}

\section{Applications to the transport equation}

Now we apply the equivalent system analysis and the equivalent equation
analysis to the simple scalar transport equation (thus $m=1$)
\[
\partial_{t}w+\sum_{i=1}^{d}\partial_{i}(v_{i}w)=0,
\]
where the velocity vector $(v_{1},\ldots, v_{d})$ is supposed
to be constant. We consider the cases $d=1$ or $d=2$ and the D1Q2,
D2Q3 and D2Q4 models.

For each model, we compute the equivalent system. The expansion of
\[
\frac{\mathcal{S}(\Delta t)-\mathcal{S}^{-1}(\Delta t)}{2\Delta t}Y(\cdot,t)=\partial_{t}Y(\cdot,t)+O(\Delta t^{2}),
\]
is performed with the Computer Algebra System (CAS) Maple. This is
done by entering the explicit definition of the symmetric operator
given in Section \ref{subsec:op_notations}, step by step. Without
a CAS, the calculations would be extremely tedious... 

\subsection{D1Q2}

For the D1Q2 model, we use the notations
\global\long\def\om{\omega}%
\global\long\def\Dt{\Delta t}%
\global\long\def\y{y}%
\global\long\def\lam{\lambda}%
\global\long\def\w{w}%
\global\long\def\cc{v}%
\[
Y=\left(\begin{array}{c}
w\\
y
\end{array}\right),\quad v_{1}=v.
\]
The equivalent system reads 
\begin{equation}
\begin{split} & \partial_{t}\begin{pmatrix}w\\
y
\end{pmatrix}-\frac{\om(\om-2)(\om^{2}-2\om+2)}{2\Dt(\om-1)^{2}}\begin{pmatrix}0\\
\y
\end{pmatrix}+\begin{pmatrix}\cc & \gamma_{1}\\
(\lam^{2}-\cc^{2})\gamma_{1} & \frac{-\cc(\om^{4}-4\om^{3}+6\om^{2}-4\om+2)}{2(\om-1)^{2}}
\end{pmatrix}\partial_{x}\begin{pmatrix}\w\\
\y
\end{pmatrix}\\
 & +\begin{pmatrix}-(\om^{2}-6\om+6)(\lam^{2}-\cc^{2}) & 3\cc(\om^{2}-2\om+2)\\
3\cc(\lam^{2}-\cc^{2})(\om^{2}-2\om+2) & -5\cc^{2}\om^{2}-3\lam^{2}\om^{2}+6\cc^{2}\om+10\lam^{2}\om-6\cc^{2}-10\lam^{2}
\end{pmatrix}\\
 & \times\frac{\Dt\om(\om-2)}{32(\om-1)^{2}}\partial_{xx}\begin{pmatrix}\w\\
\y
\end{pmatrix}=O(\Dt^{2}),
\end{split}
\label{eq:eq_sys_S_D1Q2_ordre_2}
\end{equation}
with $\gamma_{1}=\frac{(\om-2)^{2}(\om^{2}-2\om+2)}{8(\om-1)^{2}}.$
We can check that when $\om=2$, we indeed obtain an equivalent system
with the simplification (\ref{eq:om2_rB}) and (\ref{eq:om2_A}).
The stiff term vanishes and the evolution of $w$ is uncoupled from
that of $y$ at order 2. This means that the consistency is achieved
even when the flux error $y$ is large.

For obtaining the equivalent equation we assume that we have $\y=O(\Dt)$.
Let us write $\y=\Dt\tilde{\y}$. We obtain 
\[
\frac{\om(\om-2)(\om^{2}-2\om+2)}{2(\om-1)^{2}}\tilde{\y}=\frac{(\lam^{2}-\cc^{2})(\om-2)^{2}(\om^{2}-2\om+2)}{8(\om-1)^{2}}\partial_{x}\w+O(\Dt).
\]
By simplifying, we have 
\begin{equation}
\y=\frac{(\lam^{2}-\cc^{2})(\om-2)}{4\om}\Dt\partial_{x}\w+O(\Dt^{2}).
\label{eq:relation_y_dxw}
\end{equation}
By reinjecting this expression of $\y$ in the first equation of equivalent
system \ref{eq:eq_sys_S_D1Q2_ordre_2}, we obtain the equivalent equation
on $\w$ 
\begin{align*}
 & \partial_{t}\w+\cc\partial_{x}\w+\frac{(\om-2)^{2}(\om^{2}-2\om+2)}{8(\om-1)^{2}}\frac{(\lam^{2}-\cc^{2})(\om-2)}{4\om}\Dt\partial_{xx}\w\\
 & -\frac{\Dt\om(\om-2)}{32(\om-1)^{2}}(\om^{2}-6\om+6)(\lam^{2}-\cc^{2})\partial_{xx}\w=O(\Dt^{2}),
\end{align*}
which can be simplified in 
\begin{equation}
\partial_{t}\w+\cc\partial_{x}\w=\frac{1}{2}\left(\frac{1}{\om}-\frac{1}{2}\right)(\lam^{2}-\cc^{2})\Dt\partial_{xx}\w+O(\Dt^{2}).\label{eq:eq_eq_S_D1Q2_ordre_2}
\end{equation}
\modifa{We can notice that we recover the equivalent equation given in \cite{dubois2008equivalent,graille2014approximation,courtes2020vectorial}}.


\begin{thm}
\label{prop:stability_D1Q2}
When \modifa{$1\leq \om < 2$}, the sub-characteristic diffusive stability condition of the $D1Q2$ model is
$$
|v|<\lam.
$$
\end{thm}

\begin{proof}
The equivalent equation on $\w$ of the $D1Q2$ model \eqref{eq:eq_eq_S_D1Q2_ordre_2} is stable if the diffusion term is positive.
As $\om \in [1,2)$, the term $\left( \frac{1}{\om} - \frac{1}{2} \right)$ is positive.
The positivity of the diffusion term is then equivalent to
$$\lam^2 - v^2 > 0, $$
which gives us the stability condition $$|\cc | < \lam.$$
\end{proof}

\begin{rem}
When $\om=2$, the diffusion term of the equivalent equation of the $D1Q2$ model disappears, which gives us
$$
\partial_t \w + \cc\partial_x \w 
=O(\Dt^2).
$$
We obtain that the solution given by the $D1Q2$ model is an approximation of order $2$ of the solution of the initial equation.
\end{rem}

\begin{thm}
\label{prop:hyperbolicity_D1Q2}
The matrix $$
P= \begin{pmatrix}
1 & 0 \\ 0 & \frac{1}{\lam^2 - v^2}
\end{pmatrix},
$$
symmetrizes the equivalent system of the $D1Q2$ model \eqref{eq:eq_sys_S_D1Q2_ordre_2}, if the  diffusive
sub-characteristic stability condition
is satisfied. Consequently, the equivalent system \eqref{eq:eq_sys_S_D1Q2_ordre_2} is hyperbolic if
$$
|v| < \lam.
$$
\end{thm}

\begin{proof}
We search a matrix $P=\begin{pmatrix}
\Pu & \Pv \\ \Pv & \Pw
\end{pmatrix}$
such as $P A$ is symmetric and $P$ is symmetric positive definite. We have 

\begin{align*}
P A
&=
\begin{pmatrix}
\Pu & \Pv \\ \Pv & \Pw 
\end{pmatrix}
\begin{pmatrix}
\cc&
\gamma_1
\\
(\lam^2-\cc^2)\gamma_1
&
-\cc \gamma_2
\end{pmatrix},
\\&=
\begin{pmatrix}
v \Pu +  (\lam^2 - v^2) \gamma_1 \Pv &
\gamma_1 \Pu - v  \gamma_2 \Pv \\
v \Pv + (\lam^2 - v^2) \gamma_1 \Pw &
 \gamma_1 \Pv - v \gamma_2 \Pw
\end{pmatrix}
.
\end{align*}
with 
$\gamma_1=\frac{(\om-2)^2(\om^2-2\om+2)}{8(\om-1)^2},$
and
$\gamma_2=
\frac{(\om^4-4\om^3+6\om^2-4\om+2)}{2(\om-1)^2}.$
As we want $P A$ to be symmetric, we need to satisfy the condition
$$
\gamma_1 \Pu - v  \gamma_2 \Pv =
v \Pv + (\lam^2 - v^2) \gamma_1 \Pw,
$$
which is equivalent to 
\begin{align*}
\Pw &= \frac{1}{(\lam^2 - v^2)} \Pu - v \frac{1+ \gamma_2}{(\lam^2 - v^2) \gamma_1 } \Pv.
\end{align*}
Let us choose $\Pv=0$ and $\Pu=1$. We obtain
$$
P=\begin{pmatrix}
1 & 0 \\ 0 & \frac{1}{\lam^2 - v^2}
\end{pmatrix}.
$$
As its eigenvalues are $1$ and $\frac{1}{\lam^2 - v^2}$, $P$ is definite positive if 
$$
|v| < \lam.
$$
\end{proof}

\begin{rem}
We obtain the same condition on $v$ and $\lam$ as for the diffusive stability condition given in Proposition \ref{prop:stability_D1Q2}. In this case, the diffusive analysis and the hyperbolicity analysis give the same stability condition.
\end{rem}


\subsection{D2Q3}

\global\long\def\a{a}%
\global\long\def\b{b}%
For the D2Q3, we use the notations
\[
Y=\left(\begin{array}{c}
w\\
y_{1}\\
y_{2}
\end{array}\right),\quad v_{1}=a,\quad v_{2}=b.
\]
The equivalent system of the $D2Q3$ model is
\begin{equation}
\begin{split}\partial_{t}\begin{pmatrix}\w\\
\y_{1}\\
\y_{2}
\end{pmatrix} & -\frac{\om(\om-2)(\om^{2}-2\om+2)}{2\Dt(\om-1)^{2}}\begin{pmatrix}0\\
\y_{1}\\
\y_{2}
\end{pmatrix}\\
 & +\begin{pmatrix}\a & -2\gamma_{1} & 0\\
\gamma_{1}(2\a+\lam)(\a-\lam) & \gamma_{2}(2\a-\lam) & 0\\
\gamma_{1}\b(2\a+\lam) & 2\b\gamma_{2} & \gamma_{2}\lam
\end{pmatrix}\partial_{1}\begin{pmatrix}\w\\
\y_{1}\\
\y_{2}
\end{pmatrix}\\
 & +\begin{pmatrix}\b & 0 & -2\gamma_{1}\\
\gamma_{1}\b(2\a+\lam) & 0 & \gamma_{2}(2\a+\lam)\\
\gamma_{1}(\a\lam+2\b^{2}-\lam^{2}) & \gamma_{2}\lam & 2\b\gamma_{2}
\end{pmatrix}\partial_{2}\begin{pmatrix}\w\\
\y_{1}\\
\y_{2}
\end{pmatrix}=O(\Dt),
\end{split}
\label{eq:eq_sys_S_D2Q3}
\end{equation}
with $\gamma_{1}=-\frac{1}{16}\frac{(\om^{2}-2\om+2)(\om-2)^{2}}{(\om-1)^{2}}$
and $\gamma_{2}=-\frac{1}{4}\frac{\om^{4}-4\om^{3}+6\om^{2}-4\om+2}{(\om-1)^{2}}$.

For getting the equivalent equation, we assume that $(y_{1},y_{2})=O(\Dt)$.
As above, we express the flux error in function of the gradient of
$w$ up to order 2. This gives: 
\[
\y_{1}=\Dt\left(\frac{1}{\om}-\frac{1}{2}\right)\frac{\left(2\a+\lam\right)}{2}\left(\left(\a-\lam\right)\partial_{1}\w+\b\partial_{2}\w\right)+O(\Dt^{2}),
\]
and 
\[
\y_{2}=\frac{\Dt}{2}\left(\frac{1}{\om}-\frac{1}{2}\right)\left(\left(2\a\b+\b\lam\right)\partial_{1}\w+\left(\lam\a+2\b^{2}-\lam^{2}\right)\partial_{2}\w\right)+O(\Dt^{2}).
\]
We reinject these expressions of $\y_{1}$ and $\y_{2}$ in the first
equation of the equivalent system (\ref{eq:eq_sys_S_D2Q3}), we obtain
\begin{equation}
\partial_{t}\w+\partial_{1}(a\w) +\partial_{2}(b\w)  =\frac{\Dt}{2}\left(\frac{1}{\om}-\frac{1}{2}\right)\nabla\cdot(\mathcal{D}_{3}\nabla\w)+O(\Dt^{2}),\label{eq:eq_eq_S_D2Q3}
\end{equation}
with the diffusion matrix 
\[
\mathcal{D}_{3}=\begin{pmatrix}\frac{\lam}{2}(\lam+\a)-\a^{2} & -\frac{\lam}{2}\b-\a\b\\
-\frac{\lam}{2}\b-\a\b & \frac{\lam}{2}(\lam-\a)-\b^{2}
\end{pmatrix}.
\]
%

\begin{thm}
\label{prop:stability_D2Q3}
The sub-characteristic stability condition of the $D2Q3$ model is
\begin{equation}\label{eq:stab_d2q3}
\lam^2 -\a^2-\b^2
- \sqrt{
(\a^2+\b^2)^2 + \lam (-2\a^3+6\a\b^2) +\lam^2(\a^2+\b^2)
} > 0.
\end{equation}
\end{thm}


\begin{rem}
  \modifa{This condition has a geometric interpretation. It also states that the velocity vector $(a,b)$ has to be inside the triangle formed by the kinetic velocities, which are here $(\lambda,0)$,
  $(-\lambda/2,\lambda\sqrt{3}/2)$ and $(-\lambda/2,-\lambda\sqrt{3}/2)$.
It is not easy to guess it from the inequality (\ref{eq:stab_d2q3}), but it becomes obvious if we 
plot numerically the stability region (see Figure \ref{fig:stab_d2q3}).}
\end{rem}

\begin{proof}

Indeed, with a linear flux, we have
\begin{equation*}
\partial_t \w
+ \partial_{1}(a\w) +\partial_{2}(b\w)
=
\frac{\Dt}{2} \left( \frac{1}{\omega}-\frac{1}{2} \right)
\nabla \cdot (\mathcal{D}_3 \nabla \w)
+O({\Dt}^2),
\end{equation*}
with the diffusion matrix
$$
\mathcal{D}_3
=
\begin{pmatrix}
\frac{\lam}{2}(\lam+\a)- \a^2
& -\frac{\lam}{2} \b - \a\b \\
-\frac{\lam}{2} \b -\a\b
& \frac{\lam}{2}(\lam-\a) -\b^2
\end{pmatrix}.
$$
The eigenvalues of this diffusion matrix are 
$$
d_{1,2}=\frac{1}{2} \left(
\lam^2 -\a^2-\b^2
\pm \sqrt{
(\a^2+\b^2)^2 + \lam (-2\a^3+6\a\b^2) +\lam^2(\a^2+\b^2)
}
\right).
$$
Finally, the model $D2Q3$ is stable if $\mathcal{D}_3$ is positive definite, namely if $d_1 > 0$ and $d_2>0$.
\end{proof}

\begin{thm}
\label{prop:hyperbolicity_D2Q3}
The matrix $$
P=
\begin{pmatrix}
\frac{\lam}{2}{(\a^2-2\a\lam-3\b^2+\lam^2)(2\a+\lam)}
& 0 & 0
\\
0 & -{(\a\lam+2\b^2-\lam^2)} & \b(2\a+\lam) \\
0 & \b(2\a+\lam) & -{(\a-\lam) (2\a+\lam)}
\end{pmatrix}.
$$
symmetrizes the equivalent system of the $D2Q3$ model \eqref{eq:eq_sys_S_D2Q3}, if the diffusive
sub-characteristic stability condition \eqref{prop:stability_D2Q3}
is verified. Consequently, the equivalent system \eqref{eq:eq_sys_S_D2Q3} is hyperbolic if
$$
\lam^2 -\a^2-\b^2
- \sqrt{
(\a^2+\b^2)^2 + \lam (-2\a^3+6\a\b^2) +\lam^2(\a^2+\b^2)
} > 0.
$$
\end{thm}

\begin{proof}
~

We are searching for a matrix 
$P=\begin{pmatrix}
\Pu  & \Pv  & \Pw   \\
\Pv  & \Px & \Py \\
\Pw   & \Py & \Pz
\end{pmatrix}$
such as $P A^1$ and $P A^2$ are symmetric and $P$ is symmetric positive definite. 
When we compute the matrices $P A^1$ and $P A^2$, the symmetry imposes $6$ equations on the unknown $\Pu,\Pv,\Pw,\Px,\Py,\Pz$. This gives us the matrix 
$$
P=
\begin{pmatrix}
\frac{\lam}{2}\frac{(\a^2-2\a\lam-3\b^2+\lam^2)\Py}{\b}
& 0 & 0
\\
0 & -\frac{(\a\lam+2\b^2-\lam^2)\Py}{\b (2 \a+\lam)} & \Py \\
0 & \Py & -\frac{(\a-\lam) \Py}{\b}
\end{pmatrix},
$$
where $\Py$ must be chosen. We choose $\Py=\b(2\a+\lam)$.
We obtain 
$$
P=
\begin{pmatrix}
\frac{\lam}{2}{(\a^2-2\a\lam-3\b^2+\lam^2)(2\a+\lam)}
& 0 & 0
\\
0 & -{(\a\lam+2\b^2-\lam^2)} & \b(2\a+\lam) \\
0 & \b(2\a+\lam) & -{(\a-\lam) (2\a+\lam)}
\end{pmatrix}.
$$
The eigenvalues of $P$ are
$$
e_1=\frac{\lam}{2} (\a^2-2\a \lam-3\b^2+\lam^2)(2\a+\lam),
$$
$$
e_2
=
\lam^2-\a^2-\b^2+\sqrt{(\a^2+\b^2)^2 + \lam (-2\a^3+6\a\b^2) +\lam^2(\a^2+\b^2)},
$$
and
$$
e_3=
\lam^2-\a^2-\b^2-\sqrt{(\a^2+\b^2)^2 + \lam (-2\a^3+6\a\b^2) +\lam^2(\a^2+\b^2)}.
$$
By noticing that $e_2>e_3$ and $e_2 e_3 = 2 e_1$, we deduce that $P$ is definite positive if $e_3>0$.
\end{proof}

\begin{rem}
The hyperbolicity condition on $\a$, $\b$ and $\lam$  is the same  as the diffusive stability condition given in the Proposition \ref{prop:stability_D2Q3}.
Here again, the
diffusive analysis and the hyperbolicity analysis are equivalent. 
\begin{figure}[ht]
\begin{center}
\includegraphics[width=0.5\textwidth]{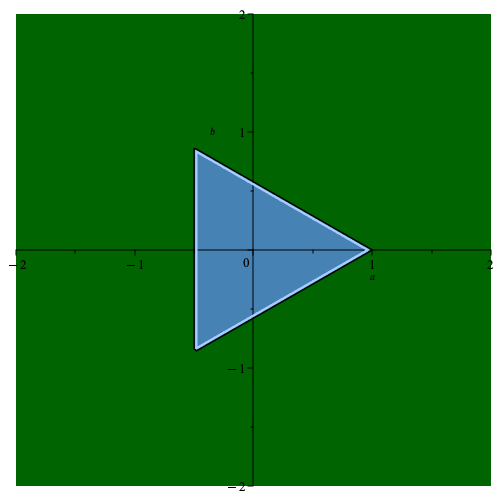}
\caption{Graphic representation of the diffusive stability and hyperbolicity condition of the $D2Q3$ model for $\lambda=1$. The stable region is the blue triangle, whose vertices are the kinetic velocities.}
\label{fig:stab_d2q3}
\end{center}
\end{figure}
\end{rem}

\subsection{D2Q4}

\global\long\def\z{z}%
For the D2Q4, we use the notations
\[
Y=\left(\begin{array}{c}
w\\
y_{1}\\
y_{2}\\
z_{3}
\end{array}\right),\quad v_{1}=a,\quad v_{2}=b.
\]
We can also compute the equivalent system on $(\w,y_{1},y_{2},z_{3})$
of the $D2Q4$ model. We obtain 
\begin{equation}
\begin{split}\partial_{t}\begin{pmatrix}\w\\
\y_{1}\\
\y_{2}\\
\z_{3}
\end{pmatrix} & -\frac{\om(\om-2)(\om^{2}-2\om+2)}{2\Dt(\om-1)^{2}}\begin{pmatrix}0\\
\y_{1}\\
\y_{2}\\
\z_{3}
\end{pmatrix}+\begin{pmatrix}
\a & 2\gamma_{1} & 0 & 0\\
\gamma_{1}(\lam^{2}-2\a^2) & -2\a\gamma_{2} & 0 & \gamma_{2}\\
-2\a\b\gamma_{1} & -2\b\gamma_{2} & 0 & 0\\
2\lam^{2}\a\gamma_{1} & 2\lam^{2}\gamma_{2} & 0 & 0
\end{pmatrix}\partial_{1}\begin{pmatrix}\w\\
\y_{1}\\
\y_{2}\\
\z_{3}
\end{pmatrix}\\
 & +\begin{pmatrix}
 \b & 0 & 2\gamma_{1} & 0\\
-2\a\b\gamma_{1} & 0 & -2\a\gamma_{2} & 0\\
\gamma_{1}(\lam^{2}-2\b^{2}) & 0 & -2\b\gamma_{2} & -\gamma_{2}\\
-2\lam^{2}\b\gamma_{1} & 0 & -2\lam^{2}\gamma_{2} & 0
\end{pmatrix}\partial_{2}\begin{pmatrix}\w\\
\y_{1}\\
\y_{2}\\
\z_{3}
\end{pmatrix}=O(\Dt),
\end{split}
\label{eq:eq_sys_S_D2Q4}
\end{equation}
with $\gamma_{1}=\frac{(\om-2)^{2}(\om^{2}-2\om+2)}{16(\om-1)^{2}}$
and $\gamma_{2}=\frac{\om^{4}-4\om^{3}+6\om^{2}-4\om+2}{4(\om-1)^{2}}$.

With the same method as above, we derive the equivalent equation 
\[
\partial_{t}\w+\partial_{1}(a\w) +\partial_{2}(b\w)=\frac{\Dt}{2}\left(\frac{1}{\om}-\frac{1}{2}\right)\nabla\cdot(\mathcal{D}_{4}\nabla\w)+O({\Dt}^{2}),
\]
with the diffusion matrix 
\[
\mathcal{D}_{4}=\begin{pmatrix}\frac{\lam^{2}}{2}-\a^{2} & -\a\b\\
-\a\b & \frac{\lam^{2}}{2}-\b^{2}
\end{pmatrix}.
\]


\begin{thm}
\label{prop:stability_D2Q4}
The $D2Q4$ model is stable if $
\a^2+\b^2 \leq \frac{\lam^2}{2}
$.
\end{thm}

\begin{proof}
We have
\begin{equation}
\partial_t \w 
+\partial_{1}(a\w) +\partial_{2}(b\w) 
= \frac{\Dt}{2} \left( \frac{1}{\om}- \frac{1}{2} \right) \nabla \cdot(\mathcal{D}_4 \nabla \w) + O(\Dt),
\end{equation}
with 
$
\mathcal{D}_4=
\begin{pmatrix}
\frac{\lam^2}{2}-\a^2
& -\a\b \\
-\a\b
&
\frac{\lam^2}{2}-\b^2
\end{pmatrix}.
$
The model is stable if the diffusion matrix $\mathcal{D}_4$ is positive. Its eigenvalues are:
\begin{align*}
&e_1=
\frac{1}{2}
\left(
\lam^2 - a^2 -b^2
-\sqrt{
\left(\lam^2 -  a^2 -b^2 \right)^2
-\lam^4
+2 \lam^2 \b^2
+2 \lam^2 \a^2
}
\right) \quad \text{ and }
\\
&
e_2= 
\frac{1}{2}
\left(
\lam^2 -  a^2 -b^2
+\sqrt{
\left(\lam^2-  a^2 -b^2 \right)^2
-\lam^4 
+2 \lam^2 \b^2
+2 \lam^2 \a^2
}
\right).
\end{align*}
As $e_1 \leq e_2$, the eigenvalues are both positive if $e_1 \geq 0$, which means if
$$
\a^2+\b^2 \leq \frac{\lam^2}{2}.
$$
\end{proof}

\begin{thm}
\label{prop:hyperbolicity_D2Q4}
The matrix $$P=
\begin{pmatrix}
\lam^2 (4\a^2-\lam^2) (4\b^2-\lam^2) & 0 & 0 & 0 \\
0 & -2 \lam^2 (4\b^2-\lam^2) & 0 & 2\a (4\b^2-\lam^2) \\
0 & 0 & -2 \lam^2 (4\a^2-\lam^2) & -2\b (4\a^2-\lam^2) \\
0 & 2\a (4\b^2-\lam^2) & -2\b (4\a^2-\lam^2) & -2\a^2-2\b^2+\lam^2
\end{pmatrix},
$$
symmetrizes the equivalent system of the $D2Q4$ model \eqref{eq:eq_sys_S_D2Q4}, if 
\begin{equation}
4 \max(\a^2,\b^2)<\lam^2.
\label{eq:hyperbolicity_D2Q4}
\end{equation}
Consequently, under this condition, the equivalent system \eqref{eq:eq_sys_S_D2Q4} is hyperbolic.
\end{thm}

\begin{proof}
We are searching for a matrix 
$$P=\begin{pmatrix}
\Pu  & \Pv  & \Pw & \Px \\
\Pv  & \Py & \Pz & \Pa\\
\Pw   & \Pz & \Pb  & \Pc\\
\Px & \Pa & \Pc & \Pp
\end{pmatrix},$$
such as $P A^1$ and $P A^2$ are symmetric and $P$ is symmetric positive definite.
We can compute $P A^1$ and $P A^2$. As we want these matrices to be symmetric, we obtain conditions on the coefficients $p_i$. 
We deduce that 
$$
P=
\begin{pmatrix}
\frac{1}{2 \a} \lam^2(2\a-\lam)(2\a+\lam) \Pa
& 0 & 0 & 0 \\
0 & - \Pa \lam^2/a & 0 & \Pa \\
0 & 0 & - \frac{\Pa \lam^2 (2\a-\lam) (2\a+\lam)}{(\a (2\b-\lam) (2\b+\lam))} & -\frac{\b  \Pa (2\a-\lam) (2\a+\lam)}{(\a (2\b-\lam) (2\b+\lam)) }\\
0 & \Pa & -\frac{\b  \Pa (2\a-\lam) (2\a+\lam)}{(\a (2\b-\lam) (2\b+\lam))}
& -\frac{1 (2\a^2+2\b^2-\lam^2)  \Pa}{2(\a (2\b-\lam) (2\b+\lam))}
\end{pmatrix}.
$$
By choosing
$\Pa = 2\a (2\b-\lam) (2\b+\lam)$, we obtain
\begin{equation}
P=
\begin{pmatrix}
\lam^2 (4\a^2-\lam^2) (4\b^2-\lam^2) & 0 & 0 & 0 \\
0 & -2 \lam^2 (4\b^2-\lam^2) & 0 & 2\a (4\b^2-\lam^2) \\
0 & 0 & -2 \lam^2 (4\a^2-\lam^2) & -2\b (4\a^2-\lam^2) \\
0 & 2\a (4\b^2-\lam^2) & -2\b (4\a^2-\lam^2) & -2\a^2-2\b^2+\lam^2
\end{pmatrix}.
\label{eq:hyperbolicity_matrix_P_D2Q4}
\end{equation}
As $P$ is symmetric, according to the Sylvester's criterion, $P$ is 
positive definite if and only if all the leading principal minors are positive, that is to say if the following conditions are satisfied
$$
\left\{
\begin{array}{ccccc}
|P_1| &=& \lam^2 (4\b^2-\lam^2) (4\a^2-\lam^2) &>& 0,
\\
|P_2| &=& -2 \lam^4  (4\a^2-\lam^2) (4\b^2-\lam^2)^2 &>& 0,
\\
|P_3| &=& 4 \lam^6  (4\a^2-\lam^2)^2 (4\b^2-\lam^2)^2 &>& 0,
\\
|P_4| &=& 4 \lam^4 (4\a^2-\lam^2)^3 (4\b^2-\lam^2)^3 & >&0.
\end{array}
\right.
$$
This is equivalent to 
$$
\left\{
\begin{array}{ccc}
4\a^2  &<& \lam^2,
\\
4\b^2& <&\lam^2,
\end{array}
\right.
$$
which can be rewritten
$$2 \max{(|\a|,|\b|)} < \lam.$$
\end{proof}

\begin{rem}
The hyperbolicity condition obtained is more restrictive than the diffusive stability condition obtained in Proposition \ref{prop:stability_D2Q4}. 
We can see in Figure \ref{fig:conditions_D2Q4} the values of $\a/\lam$ and $\b/\lam$ for which the diffusive stability condition is verified, the circle colored in yellow, are included in the blue square, for which the hyperbolicity condition is checked.
This is coherent with the review of stability conditions given by Bouchut in \cite{bouchut2005stability}.
\begin{figure}[ht]
\begin{center}
\includegraphics[width=0.4\textwidth]{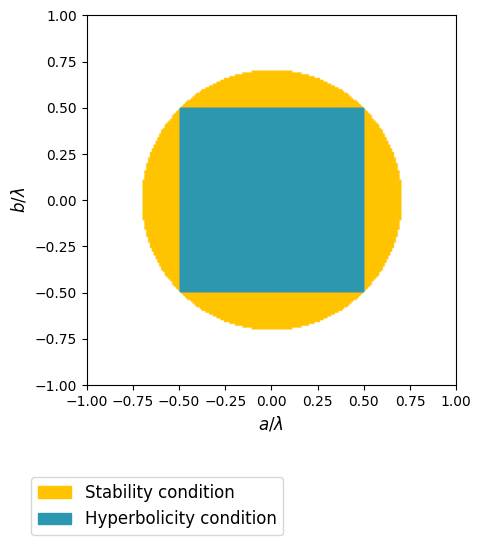}
\caption{Graphic representation of the diffusive stability and hyperbolicity condition of the $D2Q4$ model.}
\label{fig:conditions_D2Q4}
\end{center}
\end{figure}
\end{rem}

\section{Some numerical results}
\subsection{D1Q2: consistency error}

\global\long\def\gam{\gamma}%
\global\long\def\eqeq{\text{eqeq}}%
\global\long\def\eqsys{\text{syseq}}%

Now that we have obtained the equivalent equations, we wish to quantify numerically how they are close to the kinetic equations.
We shall compute analytic solutions of the equivalent equations and compare them with the solutions of the kinetic equation. We shall also compute the error between the two solutions.

For the numerical experiments, we use particular solutions of the
form 
\begin{equation}
\begin{pmatrix}\w\\
\y
\end{pmatrix}(x,t)=\begin{pmatrix}\w_{0}\\
\y_{0}
\end{pmatrix}e^{\gam t}e^{ikx},\label{eq:particular_solution_form}
\end{equation}
with $k\in\mathbb{N}$ and $\gam\in\mathbb{C}$.

\subsubsection{Particular solution of the equivalent equation}
If we inject this particular solution \eqref{eq:particular_solution_form} in the equivalent equation \eqref{eq:eq_eq_S_D1Q2_ordre_2} on $\w$, we obtain
\begin{align*}
\gam \w 
 +  i \cc k \w
 =
 - \frac{\Dt}{2} \left(\frac{1}{\om}-\frac{1}{2} \right)
k^2 \left( \lam^2-\cc^2 \right) \w.
\end{align*}
It gives us the value of $\gam$ with respect to $k$ and $\cc$
\begin{align*}
\gam
&=-
\frac{\Dt}{2} \left(  \frac{1}{\om}-\frac{1}{2} \right)
k^2
\left( \lam^2-\cc^2\right)- i \cc k.
\end{align*}
A particular solution of the equivalent equation \eqref{eq:eq_eq_S_D1Q2_ordre_2} is then
$$
\w=\w_0 e^{-\left( \frac{\Dt}{2} \left(  \frac{1}{\om}-\frac{1}{2} \right)
k^2
\left( \lam^2-\cc^2\right)+i \cc k\right) t} e^{ikx}.
$$
In order to deal with real solutions, we compute the real part of this particular solution, that we denote $\w_{\eqeq}$ and which is still a solution of \eqref{eq:eq_eq_S_D1Q2_ordre_2}
$$
\w_{\eqeq}=\Re(\w)=
\w_0 e^{-\frac{\Dt}{2} \left(  \frac{1}{\om}-\frac{1}{2} \right)
k^2
\left( \lam^2-\cc^2\right)t}\cos\left(k(x- \cc t)\right).
$$
To compute the equivalent equation, we assume that we have the relation between $\y$ and $\partial_x \w$ given by \eqref{eq:relation_y_dxw}
\begin{equation}
\y
=
\frac{(\lam^2-\cc^2)(\om-2)}{4\om} \Dt \partial_x \w.
\label{eq:y_eqeq}
\end{equation}
We denote $\y_{\eqeq}$ the real part of $\y$
\begin{align*}
\y_{\eqeq} &=
\Re(\y),\\
&=
-\frac{(\lam^2-\cc^2)(\om-2)}{4\om} \Dt  k 
\w_0 e^{-\frac{\Dt}{2} \left(  \frac{1}{\om}-\frac{1}{2} \right)
k^2
\left( \lam^2-\cc^2\right)t}\sin\left(k(x- \cc t)\right).
\end{align*}

\subsubsection{Particular solution of the equivalent system\label{sec:slow_fast_alpha}}

Now, we inject the expression of the particular solution \eqref{eq:particular_solution_form} in the equivalent system
\eqref{eq:eq_sys_S_D1Q2_ordre_2}. We obtain 
$$
\left(
\gam I_2 + \frac{1}{\Dt} R(\om)
+  i k A(Y,\om) + \Dt k^2 B(Y,\om)
\right)
\begin{pmatrix}
\w \\ \y
\end{pmatrix}
= 0,
$$
with $R(\om)=\begin{pmatrix}
0 & 0 \\
 0 & r(\om)
\end{pmatrix}$.

The previous system admits two solutions $\gam_1(k)$ and $\gam_2(k)$ depending on $k$, which are the eigenvalues of 
$- \frac{1}{\Dt} R(\om)
-  i k A(Y,\om) - \Dt k^2 B(Y,\om)$.
We have
$$\gam_1(k)=\frac{1}{\Dt}\frac{16\om^4-64\om^3+96\om^2-64\om}{32(\om-1)^2}+O(\Dt^0),$$
and
$$\gam_2(k)= -ikD(W,\om) -\frac{D(W,\om)^2 ( \lam^2\om^2-2 )+2 \lam^2}{4\om} k^2 \Dt+O(\Dt^2).
$$
One of the solutions, $\gam_1(k)$, behaves as $O(\frac{1}{\Dt})$ when $\Dt \to 0$, and the real part of the other solution $\gam_2(k)$ behaves as $O(\Dt)$ when $\Dt \to 0$.
If we compute the particular solution \eqref{eq:particular_solution_form} with the eigenvalue $\gam_1$ in $O(\frac{1}{\Dt})$, we observe that $\y$ decreases rapidly toward $0$.
If we consider instead, the solution given by the second eigenvalue $\gam_2$, $\y$ stays small and has slower variations. We choose to keep this eigenvalue $\gam_2$ for a relevant comparison with the expected behavior.

A particular solution of the equivalent system \eqref{eq:eq_sys_S_D1Q2_ordre_2} is then
$$\begin{pmatrix}
\w \\ \y
\end{pmatrix}
(x,t)
=
\begin{pmatrix}
\w_0 \\ \y_0
\end{pmatrix} 
e^{\gam_2 t} e^{ikx}.$$
To test, we compute the real part of $\w$, that we denote $\w_{\eqsys}$
$$
\w_{\eqsys} =
\Re(\w) = \w_0 e^{\Re(\gam_2)t}\cos\left(\Im(\gam_2)t+kx\right),$$
and the real part of $\y$, denoted by $\y_{\eqsys}$
$$
\y_{\eqsys}=
\Re(\y)
=
e^{\Re(\gam_2)t}
\left( \Re(\y_0) \cos\left(\Im(\gam_2)t+kx\right)
- \Im(\y_0)\sin\left(\Im(\gam_2)t+kx\right) \right).
$$

\subsubsection{Numerical comparison of \(w\)}

We take $k=2$. We choose $\w_0=1$, and we take $\y_0$ such as $(\w_0,\y_0)$ belong to the kernel of the matrix $- \frac{1}{\Dt} R(\om)
-  i k A(Y,\om) - \Dt k^2 B(Y,\om)$.

We denote $\w_{\LB}$ the solution given by the $D1Q2$ model with the initialization
\begin{equation*}
\begin{pmatrix}
\w_{\LB} \\ \y_{\LB}
\end{pmatrix}
(x,0)
=
\begin{pmatrix}
\w_0 \\ \y_0
\end{pmatrix} \cos(kx).
\end{equation*}

We compute the relative $L^2$ error between 
the solution of the equivalent equation $w_{\eqeq}$
and
the solution given by the $D1Q2$ model $\w_{\LB}$
at the final time
$$
\sqrt{
\displaystyle \frac{\displaystyle \sum_{i=0}^{Nx}  \left(\w_{\LB}^{i,Nt}-\w_{\eqeq}^{i,Nt}\right)^2}
{\displaystyle \sum_{i=0}^{Nx}  \left(\w_{\LB}^{i,Nt}\right)^2}},
$$
and the relative $L^2$ error between 
the solution of the equivalent system $w_{\eqsys}$
and $\w_{\LB}$
$$
\sqrt{
\displaystyle \frac{\displaystyle \sum_{i=0}^{Nx}  \left(\w_{\LB}^{i,Nt}-\w_{\eqsys}^{i,Nt}\right)^2}
{\displaystyle \sum_{i=0}^{Nx} \left(\w_{\LB}^{i,Nt}\right)^2}}.
$$ 
We compute the solution for different amounts of time steps
$Nt=16,32,64,128,256,512,1024$ and $2048$, which gives us different time steps $\Dt=\frac{T}{Nt}$, with $T=\pi$.

We obtain the relative errors of Figure \ref{Comparaison_eq_sys_D1Q2_err_w}, for different relaxation parameters $\omega$.

\begin{table}[ht!]
   \begin{tabular}{ccc}
\hline
$\om=2$ & $\om=1.9$ & $\om=1.8$ \\ 
$     \vcenter{\hbox{\includegraphics[width=0.3\textwidth]{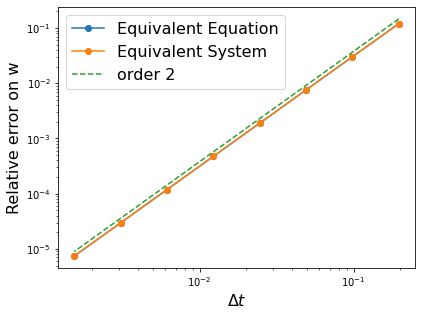}}}$ & 
 $    \vcenter{\hbox{\includegraphics[width=0.3\textwidth]{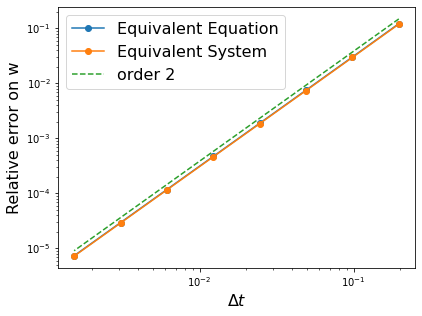}}}$ & 
   $  \vcenter{\hbox{\includegraphics[width=0.3\textwidth]{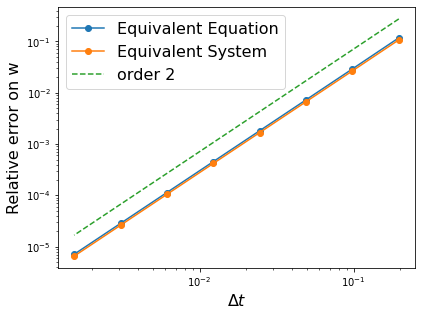}}}$
     \\ 
\hline
$\om=1.7$ & $\om=1.6$ & $\om=1.5$ \\ 
   $  \vcenter{\hbox{\includegraphics[width=0.3\textwidth]{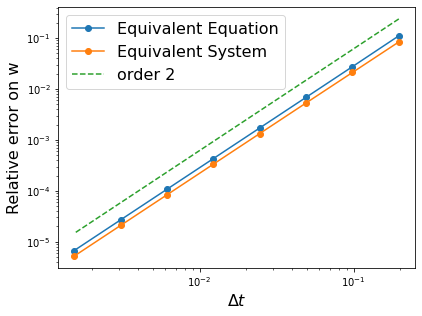}}} $& 
    $ \vcenter{\hbox{\includegraphics[width=0.3\textwidth]{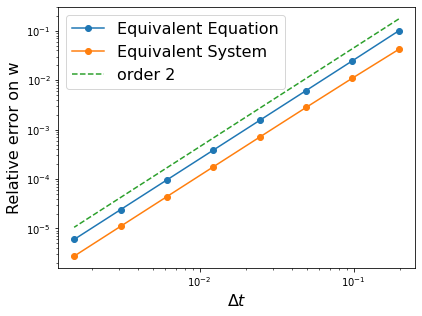}}} $& 
    $ \vcenter{\hbox{\includegraphics[width=0.3\textwidth]{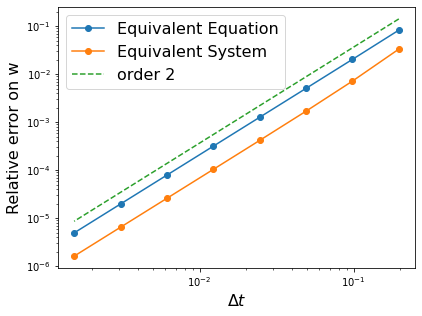}}}$
     \\ 
\hline
$\om=1.4$ & $\om=1.3$ & $\om=1.2$ \\ 
     $\vcenter{\hbox{\includegraphics[width=0.3\textwidth]{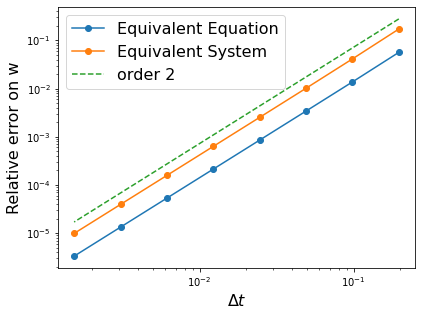}}}$ & 
   $  \vcenter{\hbox{\includegraphics[width=0.3\textwidth]{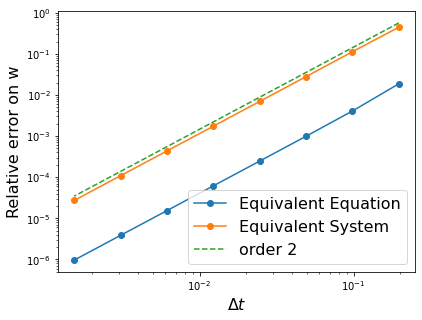}}} $& 
 $    \vcenter{\hbox{\includegraphics[width=0.3\textwidth]{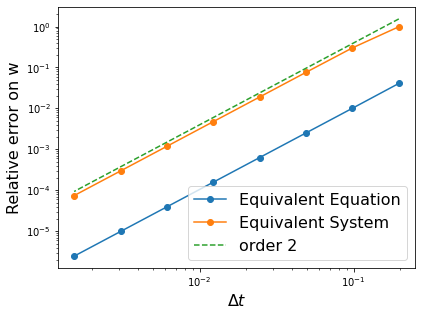}}}$
     \\ \hline
   \end{tabular}
   \caption{Relative $L^2$ error on $\w$ with respect to the time step $\Dt$, for different relaxation parameters $\om$.}
\label{Comparaison_eq_sys_D1Q2_err_w}
 \end{table}

The equivalent equation and the equivalent system both converge at the order $2$ toward the solution given by the $D1Q2$ model.
When $\om\in[1.8,2]$, the equivalent equation and the equivalent system give similar accuracy.
When $\om\in [1.5,1.8]$, the equivalent system is a better approximation of the solution given by the $D1Q2$ model, while when $\om \leq 1.4$, the equivalent equation is more accurate.

\subsubsection{Numerical comparison of y}

Now, we want to compute the error on the flux error $\y$.




We can compute the relative $L^2$ errors between $\y_{LB}$ and the  flux error $\y_{\eqeq}$ that we assume to have in order to compute the equivalent equation and between $\y_{LB}$ and the solution of the equivalent system $\y_{\eqsys}$
$$
\sqrt{
\displaystyle \frac{\displaystyle \sum_{i=0}^{Nx} 
 \left(\y_{\LB}^{i,Nt}-\y_{\eqeq}^{i,Nt}\right)^2}
{\displaystyle \sum_{i=0}^{Nx}  \left(\y_{\LB}^{i,Nt}\right)^2}}
\quad \quad \quad 
\text{ and } 
\quad \quad \quad
\sqrt{
\displaystyle \frac{\displaystyle \sum_{i=0}^{Nx} \left(\y_{\LB}^{i,Nt}-\y_{\eqsys}^{i,Nt}\right)^2}
{\displaystyle \sum_{i=0}^{Nx}  \left(\y_{\LB}^{i,Nt}\right)^2}}.
$$

We obtain the Figure \ref{Comparaison_eq_sys_D1Q2_err_y}. We can observe that the flux error $\y_{\eqeq}$ given by the equivalent equation converges at the order $1$ toward the $\y_{\LB}$ given by the $D1Q2$ model, while the $\y_{\eqsys}$ given by the equivalent system converges at the order $2$.

\begin{table}[ht!]
   \begin{tabular}{ccc}
\hline
$\om=2$ & $\om=1.9$ & $\om=1.8$ \\ 
$     \vcenter{\hbox{\includegraphics[width=0.3\textwidth]{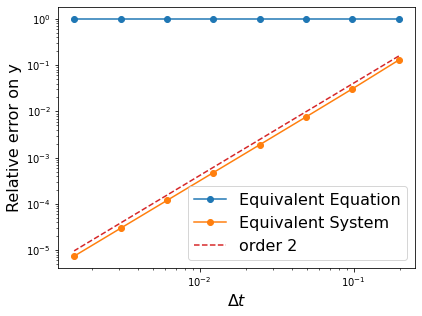}}}$ & 
 $    \vcenter{\hbox{\includegraphics[width=0.3\textwidth]{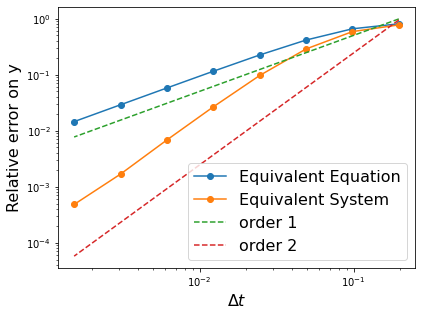}}}$ & 
   $  \vcenter{\hbox{\includegraphics[width=0.3\textwidth]{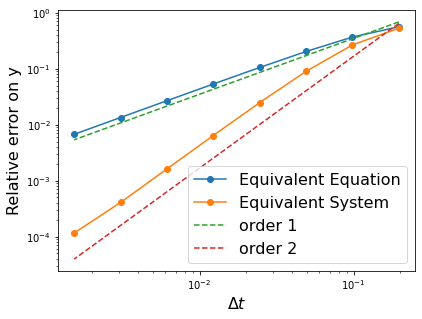}}}$
     \\ 
\hline
$\om=1.7$ & $\om=1.6$ & $\om=1.5$ \\ 
   $  \vcenter{\hbox{\includegraphics[width=0.3\textwidth]{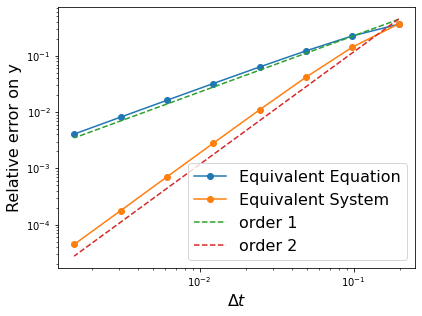}}} $& 
    $ \vcenter{\hbox{\includegraphics[width=0.3\textwidth]{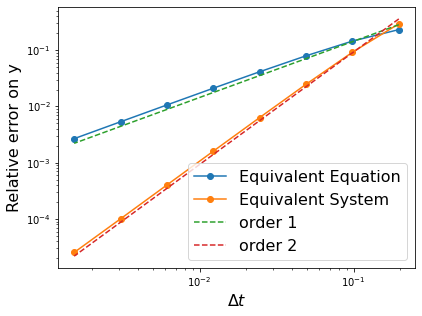}}} $& 
    $ \vcenter{\hbox{\includegraphics[width=0.3\textwidth]{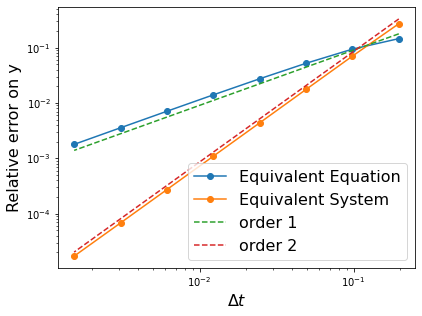}}}$
     \\ 
\hline
$\om=1.4$ & $\om=1.3$ & $\om=1.2$ \\ 
     $\vcenter{\hbox{\includegraphics[width=0.3\textwidth]{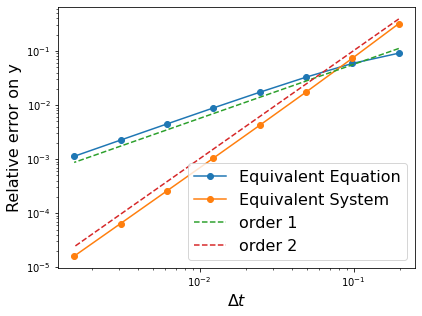}}}$ & 
   $  \vcenter{\hbox{\includegraphics[width=0.3\textwidth]{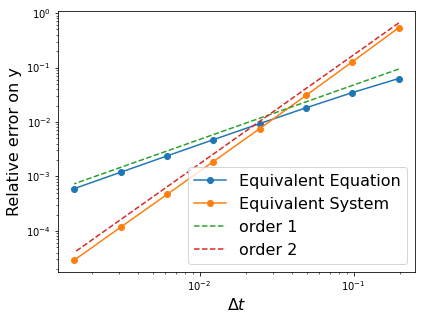}}} $& 
 $    \vcenter{\hbox{\includegraphics[width=0.3\textwidth]{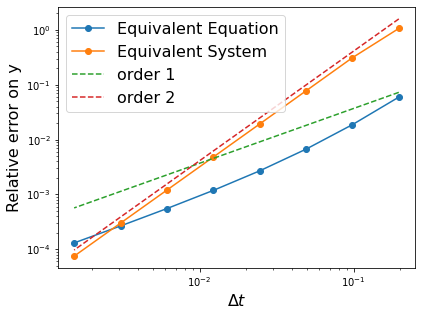}}}$
     \\ \hline
   \end{tabular}
   \caption{Relative $L^2$ error on $\y$ with respect to the time step $\Dt$, for different relaxation parameters $\om$.}
\label{Comparaison_eq_sys_D1Q2_err_y}
 \end{table}

\begin{rem}
When $\om=2$, the error between the flux error $\y$ given by the equivalent equation and the one given by the $D1Q2$ model is constant. This is due to the fact that $\y_{\eqeq}$ is given by 
\eqref{eq:y_eqeq}, which is equal to $0$ when $\omega=2$.
Indeed, when $\om=2$, $\w$ and $\y$ are independent, so we do not have to assume the smallness hypothesis $\y=O(\Dt)$ to deduce the equivalent equation from the equivalent system.
\end{rem}

\subsection{D2Q4: numerical stability}

As we can see in Figure \ref{fig:conditions_D2Q4}, for some choice
of velocity $(v_{1},v_{2})=(\a,\b)$ and norm of the kinetic velocity
$\lam$, the diffusive stability condition can be satisfied, but not
the hyperbolicity condition. We want to test numerically what happened
when we are in this case.

We consider a square geometry $[0,1]\times[0,1]$ with periodic boundary
conditions. We consider $Nx=200$ space steps in both directions.
We denote by $\Delta x=1/Nx$ the grid step. We initialized $\w$
with a Gaussian function centered in the middle of the square 
\[
\w(x,y,0)=e^{-80((x-0.5)^{2}+(y-0.5)^{2})}.
\]
Let us choose $(a,b)=(1,0)$. The stability condition is satisfied
if 
\[
\lam>\sqrt{2(\a^{2}+\b^{2})}=\sqrt{2}.
\]
The hyperbolicity condition is satisfied if 
\[
\lam>2\max(|\a|,|\b|)=2.
\]
We are going to compare the solution obtained with $\lam=1.6$, that
is when the diffusive stability condition is satisfied, but not the
hyperbolicity condition, and $\lam=2.2$, namely when both the diffusive
stability and the hyperbolicity conditions are satisfied.

We draw the solutions $\w(x,y,T)$ at time $T=1$, for different values
of the relaxation parameter: $\om=1.2$, $\om=1.6$ and $\om=2$.

As we are solving the transport step of time step $\frac{\Dt}{4}$
with a Lattice-Boltzmann method, we need to have the relation between
the time and space step 
\[
\Dt=\frac{4\Delta x}{\lam}=\frac{4}{\lam Nx}.
\]
Consequently, the number of time step is 
\[
Nt=\frac{T}{\Dt}=\frac{\lam Nx}{4},
\]
and which depends on the $\lam$ chosen: we do $Nt=80$ time steps
when $\lam=1.6$ and $Nt=110$ steps when $\lam=2.2$.

\begin{table}[t!]
\begin{tabular}{|c|c|c|}
\hline 
 & $\lam=1.6$ & $\lam=2.2$\tabularnewline
\hline 
$\om=2$ & \includegraphics[width=0.41\textwidth]{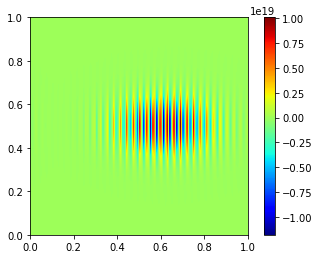} & \includegraphics[width=0.41\textwidth]{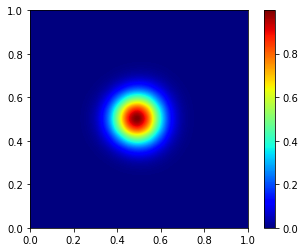}\tabularnewline
\hline 
$\om=1.6$ & \includegraphics[width=0.41\textwidth]{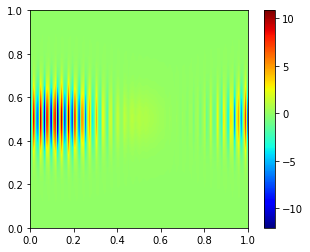} & \includegraphics[width=0.41\textwidth]{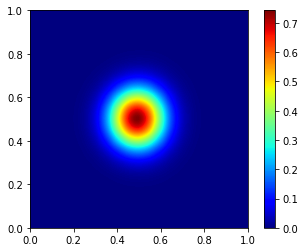}\tabularnewline
\hline 
$\om=1.2$ & \includegraphics[width=0.41\textwidth]{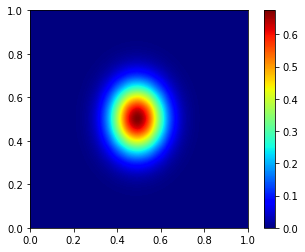} & \includegraphics[width=0.41\textwidth]{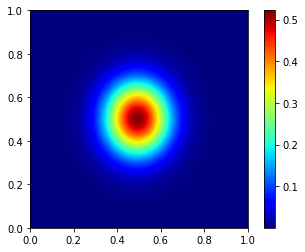}\tabularnewline
\hline 
\end{tabular}\caption{Solutions $\protect\w$ at time $T=1$ for different values of $\protect\lam$
and $\protect\om$.}
\label{tab:comparison_w_with_different_stability_conditions}
\end{table}

We obtained the solution $\w$ of the Table \ref{tab:comparison_w_with_different_stability_conditions}.
When $\lam=2.2$, that is to say when both the diffusive stability
and the hyperbolicity conditions are verified, we obtained a Gaussian
centered in the middle of the square, as expected. However, the closer
the relaxation parameter $\om$ is to $1$, the more the Gaussian
function dampens due to the relaxation step. When $\om=2$ or $\om=1.6$,
the solutions obtained with $\lam=1.6$, namely when the diffusive
stability condition is satisfied but not the hyperbolicity condition,
are not stable. Oscillations appear and grow over time. When $\om=1.2$
and $\lam=1.6$, we obtained a solution close to the expected Gaussian
function, but a little distorted. Moreover, this solution is stable,
we do not observe any oscillations.

\section{Conclusion}

In this work we have provided a general methodology for studying the
stability and the consistency of the Vectorial Lattice-Boltzmann Method
(VLBM). We have first shown that the dual entropy analysis of \cite{bouchut1999construction,dubois2014entropy}
can be applied for a direct and rigorous proof of the stability of
the over-relaxed time splitting algorithm. It is not necessary to
pass through the stiff relaxation intermediary.

Secondly, we have proposed an automatic way to construct an equivalent
system of PDE, consistent with the VLBM. This equivalent system contains
stiff terms in $\Delta t$. The classical equivalent equation can
be derived from the equivalent system by a Chapman-Enskog analysis
when $\Delta t$ is small and the kinetic data close to equilibrium.
It seems, but it is a conjecture, that the hyperbolicity condition
of the equivalent system is exactly the entropy stability condition.

In future works, we plan to investigate this conjecture, perform additional
numerical experiments on truly non-linear systems, beyond the simple
transport equation. Finally, an important question is to extend the
stability analysis for handling boundary conditions correctly (see
\cite{helie2023scheme} for preliminary results).

\bibliographystyle{plain}
\bibliography{cemracs2022_vlbm}

\end{document}